\documentclass[11pt]{article}

\usepackage{hyperref}
\usepackage{multicol}
\usepackage{amsmath,amsthm,amscd}
\usepackage{psfrag}
\usepackage{graphicx}
\normalfont\upshape

\usepackage{a4wide}
\newcommand{\pd}{\partial} 

\usepackage{bbm}
\def\C{{\mathbbm C}}

\def\R{{\mathbbm R}}
\def\Z{{\mathbbm Z}}

\def\sltwo{\mathfrak{sl}_2}
\def\osponetwo{\mathfrak{osp}_{1|2}}

\def\sym{\Lambda}

\def\nsym{N\sym}
\newcommand{\xt}{\widetilde{x}}
\newcommand{\et}{\widetilde{e}}

\newcommand{\undj}{\underline{\mathrm{J}}}
\newcommand{\symt}{\widetilde{\Lambda}}

\usepackage[left=.8in, right=.8in, top=.8in, bottom=.8in]{geometry}

\theoremstyle{definition}
\newtheorem{thm}{Theorem}[section]
\newtheorem{cor}[thm]{Corollary}

\newtheorem{lem}[thm]{Lemma}
\newtheorem{rem}[thm]{Remark}
\newtheorem{defn}[thm]{Definition}
\newtheorem{example}[thm]{Example}
\newtheorem{question}[thm]{Question}

\usepackage[vcentermath,enableskew]{youngtab}

\numberwithin{equation}{section}
\usepackage{latexsym,amssymb}

\makeatletter
\renewcommand*{\ps@plain}{%
  \let\@mkboth\@gobbletwo
  \def\@oddhead{%
    \reset@font
    \hfil
   \markright{\protect\parbox[b]{2cm}{foo\\bar}}
  }%
  \def\@oddfoot{%
    \reset@font

  }
  \let\@evenhead\@oddhead
  \let\@evenfoot\@oddfoot
}
\makeatother

\usepackage{fancyheadings}
\usepackage{tikz}
\usetikzlibrary{calc,shadings}
\usepackage{pgfplots}

\newenvironment{customlegend}[1][]{%
    \begingroup
    \csname pgfplots@init@cleared@structures\endcsname
    \pgfplotsset{#1}%
}{%
    \csname pgfplots@createlegend\endcsname
    \endgroup
}%

\def\addlegendimage{\csname pgfplots@addlegendimage\endcsname}

\pgfkeys{/pgfplots/number in legend/.style={%
        /pgfplots/legend image code/.code={%
            \node at (0.295,-0.0225){#1};
        },%
    },
}
\usetikzlibrary{decorations.pathmorphing}
\usepackage{verbatim}

\begin{document}


\begin{center}

{\huge \underline{Odd Dunkl Operators and nilHecke Algebras} \par}
\vspace{.5cm}
{\large {Ritesh Ragavender} \par}
{\large {July 14, 2014} \par}
\end{center}

\vspace{-.9cm}
\section*{Abstract}

Symmetric functions appear in many areas of mathematics and physics, including enumerative combinatorics, the representation theory of symmetric groups, statistical mechanics, and the quantum statistics of ideal gases. In the commutative (or ``even'') case of these symmetric functions, Kostant and Kumar introduced a nilHecke algebra that categorifies the quantum group $U_q(\mathfrak{sl}_2)$. This categorification helps to better understand Khovanov homology, which has important applications in studying knot polynomials and gauge theory. Recently, Ellis and Khovanov initiated the program of ``oddification'' as an effort to create a representation theoretic understanding of a new ``odd'' Khovanov homology, which often yields more powerful results than regular Khovanov homology. In this paper, we contribute towards the project of oddification by studying the odd Dunkl operators of Khongsap and Wang in the setting of the odd nilHecke algebra. Specifically, we show that odd divided difference operators can be used to construct odd Dunkl operators, which we use to give a representation of $\mathfrak{sl}_2$ on the algebra of skew polynomials and evaluate the odd Dunkl Laplacian. We then investigate $q$-analogs of divided difference operators to introduce new algebras that are similar to the even and odd nilHecke algebras and act on $q$-symmetric polynomials. We describe such algebras for all previously unstudied values of $q$. We conclude by generalizing a diagrammatic method and developing the novel method of insertion in order to study $q$-symmetric polynomials from the perspective of bialgebras. 

\vspace{-.5cm}
\tableofcontents

\section{Introduction}
\subsection{The Commutative (``Even'') Case} \label{intro1}

\emph{Symmetric polynomials} are polynomials in $n$ independent, \emph{commutative} variables $x_1,x_2,\ldots,x_n$ that are invariant under the action of any permutation acting on the indices. They arise in enumerative combinatorics, algebraic combinatorics, Galois theory, quantum statistics, and the quantum mechanics of identical particles \cite{physics,stanley}. The \emph{even nilHecke algebra} $NH_n$, introduced by Kostant and Kumar in \cite{kk}, is important in studying these symmetric polynomials. $NH_n$ is graded Morita equivalent to the symmetric polynomials in $n$ variables, and is generated by $n$ commuting variables $x_1,\ldots,x_n$ and $n$ divided difference operators $\pd_i=(x_i-x_{i+1})^{-1}(1-s_i)$, for $1 \leq i \leq n$. Here, $s_i$ is the simple transposition in the symmetric group that swaps $x_i$ and $x_{i+1}$.

Combining these divided differences with partial derivatives, one obtains a commuting family of operators originally introduced by Dunkl \cite{dunkl}.  These \emph{Dunkl operators}, denoted $\eta_i$, have a major role in mathematical physics and conformal field theory. In particular, they relate to the study of quantum many-body problems in the Calogero-Moser-Sutherland model, which describes integrable systems of one dimension \cite{etingof,rosler}. Dunkl operators can also be used to define three operators, which arise in physics and harmonic analysis, that satisfy the defining relations of the Lie algebra $\sltwo$ \cite{heckman}. These three operators, found by Heckman and called an \emph{$\sltwo$-triple}, play a crucial role in studying Fischer decomposition, which has importance not only in representation theory but also in the algebraic Dirichlet problem \cite {dirichlet,db,fischer}.

The \emph{Cherednik operators}, denoted by $Y_i$, are defined in terms of Dunkl operators and have important applications in representation theory \cite{bf, knop}. They have non-degenerate simultaneous eigenfunctions, known as \emph{Jack polynomials}. These polynomials are a specific case of the well-known Macdonald polynomials and contribute to representation theory, statistical mechanics, and the study of the quantum fractional Hall problem, important in condensed matter physics \cite{hall,knop}. 

The diagram below depicts the relationship between the three operators discussed so far.  
\vspace{.5cm}

{\centering
\begin{tikzpicture}[scale=1, every node/.style={transform shape}]
 \draw (-2.1,0.5) [black,thick] rectangle (11.3,-4.3);
  \node (D) {Divided Differences $\pd_i$};
 
  \node (CAT) [below of=D, node distance=1.2cm]{nilHecke Algebra};
  \node (hI) [below of=D, node distance=1.65cm]{Schubert polynomials};
  \node (bye) [below of=D, node distance=2.1cm]{Cohomology};

  \node (N) [right of=D,node distance=4.5 cm]{Dunkl Operators $\eta_i$};
  \node (S) [below of=N, node distance=1.2 cm]{$\sltwo$ triple};
  \node (C) [right of=N, node distance=4.5 cm]{Cherednik Operators $Y_i$};
  \node (J) [below of=C, node distance=1.2 cm]{Jack Polynomials };
  \node (CA) [below of=C, node distance=1.68 cm]{Cherednik Algebras};
  \node (AH) [below of=C, node distance=2.1 cm]{Affine Hecke Algebras};
  
  \node(HA) [below of=S, node distance=0.45 cm] {(harmonic analysis)};
  \node(HA) [below of=S, node distance=0.9 cm] {quantum CMS Model};
  
  \draw[blue,very thick][->] (D) to (N);
  \draw[orange, very thick][->] (D) to (CAT);
  \draw[orange,very thick][->] (N) to (S);
  \draw[blue,very thick][->] (N) to (C);
  \draw[orange, very thick][->] (C) to (J);

\begin{customlegend}[
legend entries={ 
\normalsize{``can be used to''},
\normalsize{``used to define''}
},
legend style={at={(6.45,-2.65)},font=\footnotesize}] 
    \addlegendimage{->,orange, very thick}
    \addlegendimage{->,blue, very thick}
\end{customlegend}
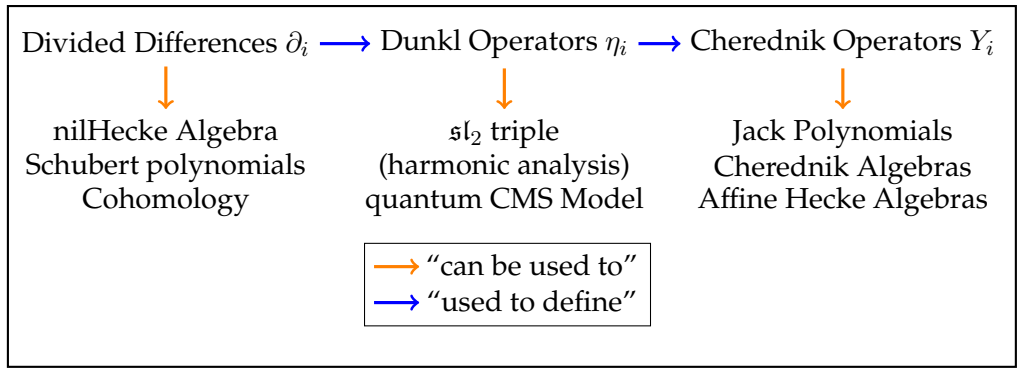
\node at (4.7,-4.8) {Figure 1: Operators in the study of symmetric polynomials.};
 \end{tikzpicture}
 \par}

\subsection{The ``Odd'' Case}
The divided difference operators, Dunkl operators, and Cherednik operators all study the \emph{commutative} symmetric functions. Ellis and Khovanov, however, sought to study different kinds of symmetric functions. They recently introduced the \emph{quantum} case of symmetric functions, where $x_jx_i=qx_ix_j$ for $j>i$ \cite{ek}. In the ``odd'' case $q=-1$, they describe the ``odd symmetric polynomials'', which are polynomials in the $n$ variables $x_1,\ldots,x_n$ where $x_ix_j+x_jx_i=0$ for $i \neq j$. This type of noncommutativity arises in the study of exterior algebras and parastatistics. 

The motivation for considering these odd symmetric polynomials and their corresponding odd nilHecke algebra involves the categorification of quantum groups. \emph{Categorification}, introduced by Crane and Frenkel, is generally the process of replacing algebras and representations by categories and higher categories in order to make quantum 3-manifold invariants into 4-manifold invariants \cite{cf}. In physics, categorification corresponds to increasing dimensions, which allows one to understand symmetries in lower dimensions and then use categorification to better understand higher dimensions. In mathematics, categorified quantum groups give a higher representation theoretic construction of \emph{link homologies}, which in turn categorify quantum link polynomials. 

The original example of link homology is \emph{Khovanov homology}, a bigraded abelian group which categorifies the well-known Jones polynomial. It has major applications in studying knot polynomials, quantum field theory, and classical gauge theory \cite{khovanov, gauge}. Since the quantum group $U_q(\sltwo)$ plays a role in understanding the Jones polynomial, a categorification of $U_q(\sltwo)$ would be useful in better understanding Khovanov homology. This precise categorification is achieved through the ``even'' nilHecke algebra $NH_n$ described in subsection \ref{intro1}.

Recently, Ozsv\'ath, Rasmussen, and Szab\'o found an \emph{odd} analog of Khovanov homology \cite{ors}. Their odd Khovanov homology also categorifies the Jones polynomial, and agrees modulo 2 with Khovanov homology. However, both theories can detect knots that the other theory cannot \cite{okh}. The subject of odd Khovanov homology has yet to be fully understood, despite its crucial connections with Khovanov homology.

Knowing this, Ellis, Lauda, and Khovanov developed the odd nilHecke algebra to provide an \emph{odd} categorification of $U_q(\sltwo)$ and give a construction of odd Khovanov homology from a representation theoretic standpoint \cite{ekl,el}. In addition to being useful in the categorification of quantum groups, the odd nilHecke algebra is also related to Hecke-Clifford superalgebras \cite{kw2,kw3} and has been used to construct odd analogs of the cohomology groups of Springer varieties \cite{lr}. 

The below diagram summarizes the categorifications that motivate the present work. NH stands for nilHecke, Cat. stands for categorification, and KH stands for Khovanov homology.
\\[11pt]
{\centering
\begin{tikzpicture}
\draw (-3.3,3.5) [black,thick] rectangle (10.3,-2.3);
\node (Jones) at (3,0){Jones Polynomial};
\node (sl21) at (7,0) {$U_q(\sltwo)$};
\node(sl22) at (-1,0) {$U_q(\sltwo)$};
\draw [->,red,very thick] (sl21) to (Jones);
\draw [->,red,very thick] (sl22) to (Jones);
\node (KH) at (1.8,3) {KH};
\node (OKH) at (4.2,3) {Odd KH};
\draw[<->,green,very thick] (KH) to (OKH);
\node(C) at (-1,3) {Cat. of $U_q(\sltwo)$};
\node(OC) at (8,3) {Odd Cat. of $U_q(\sltwo)$};
\draw[->,red,very thick] (C) to (KH);
\draw[->,red,very thick] (OC) to (OKH);
\node(odd NH) at (8.6,1.5) {odd NH algebra};
\node(odd NH) at (-2.15,1.5) {NH algebra};

\draw [->, orange,very thick,
line join=round,
decorate, decoration={
    zigzag,
    segment length=4,
    amplitude=.9,post=lineto,
    post length=2pt
}]  (-1,2.75) -- (-1,.25);

\draw [->, orange,very thick,
line join=round,
decorate, decoration={
    zigzag,
    segment length=4,
    amplitude=.9,post=lineto,
    post length=2pt
}]  (7,2.75) -- (7,.25);

\draw [->, orange,very thick,
line join=round,
decorate, decoration={
    zigzag,
    segment length=4,
    amplitude=.9,post=lineto,
    post length=2pt
}]  (1.8,2.75) -- (2.95,.25);
\draw [->, orange,very thick,
line join=round,
decorate, decoration={
    zigzag,
    segment length=4,
    amplitude=.9,post=lineto,
    post length=2pt
}]  (4.2,2.75) -- (3.05,.25);

\begin{customlegend}[
legend entries={ 
\normalsize{``categorifies''},
\normalsize{``helps to explain''},
\normalsize{``equivalent modulo 2''}
},
legend style={at={(6,-0.5)},font=\footnotesize}] 
    \addlegendimage{->,orange,very thick, line join=round,decorate, decoration={ 
    zigzag,
    segment length=4,
    amplitude=.9,post=lineto,
    post length=2pt}}
    \addlegendimage{->,red, very thick}
    \addlegendimage{<->,green,very thick}
    
\end{customlegend}
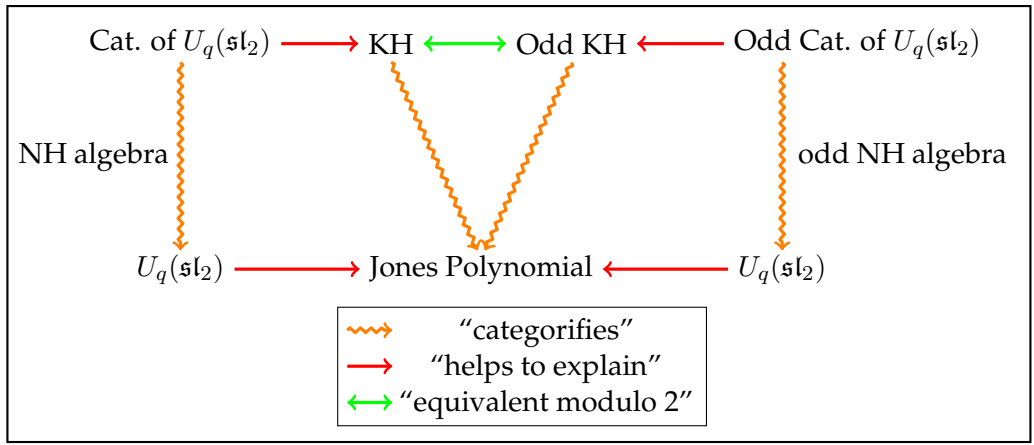
\node at (3.5,-2.8) {Figure 2: Odd Khovanov Homology and Categorification};
\end{tikzpicture}
\par}

\subsection{Outline of the Present Paper}
The main goal of the present paper is to make progress towards giving a representation theoretic construction of odd Khovanov homology. Despite being relatively new, odd Khovanov homology seems to have great importance in knot theory. It has connections to Heegaard-Floer homology and yields \emph{stronger} results than Khovanov homology in bounding the Thurston-Bennequin number and detecting quasi-alternating knots \cite{okh}. Odd Khovanov homology is also related to signed hyperplane arrangements, which have many implications in graph theory and topology \cite{hyper}.

We study odd Khovanov homology by looking for new representation theoretic structures that arise from identifying ``odd'' analogs of structures that play important algebraic roles in the even case. Ellis, Khovanov, and Lauda started this program of ``oddification'' by finding an odd analog of $NH_n$, and using it to categorify $U_q(\sltwo)$. Searching for the geometry underlying these odd constructions would also provide a very new approach to noncommutative geometry. For example, we study potential generators of certain Cherednik algebras. Since spherical rational Cherednik algebras fit nicely into a family of algebras from the geometry of symplectic resolutions, including Webster's tensor product algebras and cyclotomic KLR algebras, this project may be used in the context of ``odd'', noncommutative geometry. 

In Subsection \ref{intro1}, we discussed the (even) divided difference operators, Dunkl operators, and Cherednik operators. Although analogs of Dunkl operators have been found in the odd case, they have not been well-studied. Odd Cherednik operators have not even been defined.

As a result, the first goal of the present paper is to unify certain results in the odd case and to further study the odd Dunkl operators of Khongsap and Wang. In Section \ref{odd dunkl}, we introduce an operator $r_{i,k}$ related to the generalized odd divided difference operator $\pd_{i,k}^{\text{odd}}$ and study its properties. One of our main results (Equation \ref{etaexplicit}) is that the odd Dunkl operator $\eta_i$ may be expressed in terms of the odd divided difference operators of Ellis, Khovanov and Lauda:
\begin{equation} 
\eta_i^{\text{odd}}=t\delta_i+u \sum_{k \neq i} \pd_{i,k}^{\text{odd}}s_{i,k}.
\end{equation}
This result connects odd Dunkl operators and the odd nilHecke algebra, both of which play important roles in the project of oddification. 

In the even case, one can introduce an operator known as the \emph{Dunkl Laplacian}, given by $\sum_{i=1}^n \eta_i^2$. This operator has important applications in spherical harmonics and heat semigroups \cite{rosler}. Our next goal in the present paper is to express the Dunkl Laplacian in the odd case. In Section \ref{cybe section}, we show that the $r_{i,k}$ satisfy the classical Yang-Baxter equation, and use this result to evaluate the \emph{odd} Dunkl Laplacian. Specifically, we show that
\begin{equation*}
\sum_{i=1}^n \eta_i^2=t^2\sum_{1 \leq i  \leq n}x_i^{-2}(1-\tau_i).
\end{equation*}

In Section \ref{sl2}, we find an odd analog of Heckman's important $\sltwo$-triple in \cite{heckman} by showing that a variant $D_i$ of the odd Dunkl operator can be used to construct three operators $r^2$, $E$, and $\Delta$ that satisfy the defining relations of the Lie algebra $\sltwo$:
\begin{equation*}
r^2=(2t)^{-1}\sum_{i=1}^n x_i^2 
\end{equation*}
\begin{equation*}
E=\sum_{i=1}^n x_ip_i+\frac{n}{2}+\frac{u}{t}\sum_{k \neq i}s_{i,k}
\end{equation*}
\begin{equation*}
\Delta=-(2t)^{-1}\sum_{i=1}^n D_i^2.
\end{equation*}

Since even Dunkl operators play an important role in the representation theory of symmetric groups, our study of odd Dunkl operators should result in a better understanding of the representation theory of odd symmetric functions, which correspondingly results in a better understanding of odd Khovanov homology. 

The second goal of this paper is to study a generalization of the odd symmetric functions known as $q$-symmetric functions, for which $x_j x_i = q x_i x_j$ when $j>i$. Previous authors have described a nilHecke algebra structure only for the odd case $q=-1$ and the even case $q=1$ \cite{ekl,kk}. In Section \ref{nh}, we find a \emph{$q$-divided difference operator} for \emph{all} previously unstudied values of $q$, and explore its properties. We show, for example, that twisted elementary symmetric polynomials are in the kernel of $q$-divided difference operators, just as odd elementary symmetric functions are in the kernel of odd divided difference operators. We then use $q$-divided difference operators to construct algebras acting on $q$-symmetric polynomials that have many similarities to the even and odd nilHecke algebras. We call these \emph{$q$-nilHecke algebras}. These algebras are nontrivial generalizations of the even and odd nilHecke algebras because the $q$-twist map introduced in Section \ref{nh} is not its own inverse when $q^2 \neq 1$. In Section \ref{diagrams}, we present the elementary $q$-symmetric polynomials using a generalization of a clever diagrammatic method arising in the context of bialgebras. We use these diagrams to study relations between elementary $q$-symmetric polynomials when $q$ is a root of unity. These methods can be combined with the algebras of Section \ref{nh} in order to continue studying $q$-symmetric polynomials, including $q$-Schur and $q$-monomial functions. 

In the conclusion, we also define the odd Cherednik operators and outline a procedure for finding and studying odd analogs of Jack polynomials. This makes progress towards answering a question of Ellis about the existence of Macdonald-like polynomials in the odd case. Since the Jack polynomials have importance in representation theory, their study would enhance our knowledge about the odd algebraic theory. 

\section{Odd Dunkl operators and the Odd nilHecke algebra} \label{odd dunkl}

\subsection{Preliminaries: Even Dunkl Operators}\label{background}
In the even case, we work with the ring $\C[x_1,\ldots,x_n]$ and a root system of type $A_n$, where $x_i x_j= x_j x_i$ for all $1 \leq i,j \leq n$ and $\alpha \in \C$. We first introduce some notation involving the symmetric group.
\begin{enumerate}
\item Let $s_{i,k}$ be the simple transposition in $S_n$ swapping $x_i$ and $x_{k}$. We let $s_i=s_{i,i+1}$. 
\item Let $s_{i,j}(k,\ell)$ be the result of applying $s_{i,j}$ to the pair $(k,\ell)$. Similarly define $s_{i,j}(k)$.
\end{enumerate} 
In \cite{dunkl}, Dunkl introduced the remarkable operator
\begin{equation*}
\eta_i^{\text{even}}=\frac{\pd}{\pd x_i}+\alpha\sum_{k \neq i}\pd_{i,k}^{\text{even}},
\end{equation*}
where $\frac{\pd}{\pd x_i}$ is the partial derivative with respect to $x_i$ and $\pd_{i,k}^{\text{even}}$ is the even divided difference operator: 
\begin{equation*}
\pd_{i,k}^{\text{even}}=(x_i-x_k)^{-1}(1-s_{i,k}).
\end{equation*}
Since $x_i-x_k$ always divides $f-s_{i,k}(f)$ for $f \in \C[x_1,\ldots,x_n]$, $\pd_{i,k}$ sends polynomials to polynomials. 

These Dunkl operators have various important properties, one of which is that they commute ($\eta_i \eta_j=\eta_j\eta_i$). In \cite{kw}, Khongsap and Wang introduced anti-commuting \emph{odd} Dunkl operators on skew polynomials. In Section \ref{odd dunkl}, we will develop the connection between these operators and the odd nilHecke algebra introduced in \cite{ekl}.

Returning to the even case, introduce operators $r^2$, $E$ (the Euler operator) and $\Delta_k$: 
\begin{align*}
&r^2=\frac{1}{2}\sum_{i=1}^n x_i^2 \\
&E=\sum_{i=1}^n x_i \frac{\pd}{\pd x_i}+\frac{\mu}{2} \\
&\Delta=\frac{1}{2}\sum_{i=1}^n \eta_i^2 ,
\end{align*}
where $\mu$ is the \emph{Dunkl dimension}, which is defined by the relation $\Delta|x|^2=2\mu$ as in \cite{db}.

Let $[p,q]=pq-qp$ be the commutator. Heckman showed that $r^2$, $E$, and $\Delta_k$ satisfy the defining relations of the Lie algebra $\sltwo$ \cite{heckman}:
\begin{align*}
&[E,r^2]=2r^2 \\
&[E,\Delta_k]=-2\Delta_k \\
&[r^2,\Delta_k]=E.
\end{align*}

\begin{rem}
If one were to replace $\Delta_k$ with the classical Laplacian on flat $\R^n$ (replacing the Dunkl operator with the partial derivative), these three operators still satisfy the $\sltwo$ relations. 
\end{rem}
\begin{rem}
From now on, we will use $\eta_i$ to denote the \emph{odd} Dunkl operator of Khongsap and Wang, defined in equation \ref{kw form}.
\end{rem}
In Section \ref{sl2}, we will focus on finding analogous results in the odd case.

\subsection{Introduction to the Odd nilHecke Algebra}
We will now discuss operators with the algebra P$^-=\C\langle x_1,\ldots,x_n\rangle/\langle x_jx_i+x_ix_j=0 \text{ for } i\neq j\rangle$. We call P$^-$ the \emph{skew polynomial ring}. We can define linear operators, called the odd divided difference operators, as below: 
\begin{defn} For $i=1,\ldots.n-1$, the $i$-th \emph{odd divided difference operator} $\pd_i$ is the linear operator P$^-\to \text{ }$P$^-$ defined by $\pd_i(x_i)=1$, $\pd_i(x_{i+1})=1$, $\pd_i(x_j)=0$ for $j \neq i,i+1$, and 
\begin{equation*}
\pd_i(fg)=\pd_i(f)g+(-1)^{|f|}s_i(f)\pd_i(g),
\end{equation*}
for all functions $f,g \in$ P$^-$. We call this last relation the \emph{Leibniz rule}.
\end{defn}
It is shown in \cite{ekl} that the odd divided difference operators can be used to construct an odd nilHecke algebra, generated by $x_i$ and $\pd_i$ for $1 \leq i \leq n$, subject to the following relations:
\begin{multicols}{2}
\begin{enumerate}
    \item $\pd_i^2=0$
    \item $\pd_i\pd_j+\pd_j\pd_i=0 \text{ for } |i-j| \geq 2$
    \item $\pd_i\pd_{i+1}\pd_i=\pd_{i+1}\pd_i\pd_{i+1}$
    \item $x_ix_j+x_jx_i=0 \text{ for } i\neq j$
    \item $x_i\pd_i+\pd_ix_{i+1}=1, \pd_ix_i+x_{i+1}\pd_i=1$
    \item $x_i\pd_j+\pd_jx_i=0 \text{ for } i \neq j,j+1$.
\end{enumerate}
\end{multicols}

Due to \cite{kko}, we have the following explicit definition of the odd divided difference operator:
\begin{equation} \label{oddform}
\pd_i(f)=(x_{i+1}^2-x_i^2)^{-1}[(x_{i+1}-x_i)f-(-1)^{|f|}s_i(f)(x_{i+1}-x_i)].
\end{equation}
Although this formula \emph{a priori} involves denominators, it does take skew polynomials to skew polynomials. We extend this definition to non-consecutive indices by replacing $i+1$ with any index $k \neq i$, for $1 \leq k \leq n$, and by replacing $s_i$ with $s_{i,k}$. Equation \ref{oddform} then becomes
\begin{equation} \label{kkoform}
\pd_{i,k}(f)=(x_{k}^2-x_i^2)^{-1}[(x_{k}-x_i)f-(-1)^{|f|}s_{i,k}(f)(x_{k}-x_i)].
\end{equation}
This extended odd divided difference operator satisfies the Leibniz rule $\pd_{i,k}(fg)=\pd_{i,k}(f)g+(-1)^{|f|}s_{i,k}\pd_{i,k}(g)$ \cite{ekl}. 

\subsection{Some Operations on Skew Polynomials} \label{ri stuff}
First, we introduce a common operator in the study of Dunkl operators:
\begin{defn} Let the \emph{$(-1)$-shift operator} $\tau_i$ be the automorphism of $P^-$ which sends $x_i$ to $-x_i$ and $x_j$ to $x_j$ for $j \neq i$. \end{defn}
Suppose $1 \leq i \neq j \leq n$ and $1 \leq k \neq \ell \leq n$, where $f$ is an element in $\C\langle x_1,\ldots,x_n\rangle / \langle x_ix_j+x_jx_i=0 \text{ for } i\neq j \rangle$. Then one has
\begin{align*}
&s_{i,j}\tau_{k,\ell}=\tau_{s_{i,j}(k,\ell)} s_{i,j}\\
&fx_i=(-1)^{|f|}x_i\tau_i(f). 
\end{align*}

\begin{rem}
Since skew polynomials are not super-commutative, we cannot say that $fg = (-1)^{|f||g|} gf$.  But the operator $\tau_i$ allows us to track the discrepancy from super-commutativity, since $x_i f = (-1)^{|f|} \tau_i(f) x_i$, making it useful in this context.
\end{rem}
We now introduce the operator $r_{i,k}=\pd_{i,k}s_{i,k}$ for $k \neq i$, which will serve as another odd divided difference operator that we will use to study odd Dunkl operators. For simplicity, let $r_i=r_{i,i+1}$. In the following lemma, we study the action of the transposition and $(-1)$-shift operator on $r_{i,k}$.
\begin{lem} \label{sr} The operators $s_{i,k}$ and $\tau_i$ act on $r_{i,k}$ as follows:
\begin{equation}\label{sr exact}
r_{i,j}s_{k,\ell}=s_{k,\ell}r_{s_{k,\ell}(i,j)}.
\end{equation}
We also have that
\begin{multicols}{2}
\begin{enumerate}
\item $s_ir_{i,k}=r_{i+1,k}s_i \text{ if } k \neq i+1 $
\item $s_i r_i = r_i s_i $
\item $s_i r_{i+1}=r_{i,i+2}s_i$
\item $s_{i+1}r_i=r_{i,i+2}s_{i+1}$
\item $s_i r_j=r_j s_i \text{ for } |i-j| \geq 2  \label{s_ir_j}$
\item $\tau_i r_j = r_j \tau_i \text{ for } |i-j| \geq 2. $
\end{enumerate}
\end{multicols}
\end{lem}
\begin{proof}

Recall the following relationship between $\pd_{i,j}$ and $s_{k,\ell}$ for $i \neq j$ and $k \neq \ell$, from Lemma 2.19 (1) of \cite{ekl}: 
\begin{equation} \label{ekl lemma}
\pd_{i,j}s_{k,\ell}=s_{k,\ell}\pd_{s_{k,\ell}(i,j)}.
\end{equation}
Multiplying both sides by $s_{i,j}$, we obtain that
\begin{align*}
s_{i,j}\pd_{i,j}s_{k,\ell}&=s_{i,j}s_{k,\ell}\pd_{s_{k,\ell}(i,j)} \\
&=s_{k,\ell}s_{k,\ell}(i,j)\pd_{s_{k,\ell}(i,j)},
\end{align*}
which implies the desired result since $r_{i,j}=\pd_{i,j}s_{i,j}=s_{i,j}\pd_{i,j}$. Properties 1-5 are special cases of equation \ref{sr exact}. Property 6 follows from $\tau_i s_j = s_j \tau_i$ and the fact that $\tau_i(x_j)=x_j$ for $i \neq j$.
\begin{rem}
Differences between our formulas and those of \cite{ekl} are due to a difference of sign convention in the action of $s_{i,j}$ on $P^-$.
\end{rem}
We now show that the properties of the $r_{i,k}$ are similar to those of the odd divided difference operator $\pd_{i,k}$.
\begin{lem} The following relations hold:
\begin{multicols}{2}
\begin{enumerate}
\item $r_i^2=0$
\item $r_ir_j+r_jr_i=0 \text{ for }|i-j| \geq 2$
\item $r_ir_{i+1}r_i=r_{i+1}r_ir_{i+1}$
\item $r_{i,k}(fg)=r_{i,k}(f)s_{i,k}(g)+(-1)^{|f|}fr_{i,k}(g) \label{leibniz r_i}$
\item $r_i x_{i+1} + x_{i+1}r_i=r_ix_i+x_ir_i=s_i \label{leib1}$
\item $r_j x_i + x_j r_i =0 \text{ for }i \neq j,j+1 \label{leib3}$.
\end{enumerate}
\end{multicols}
\end{lem}
\begin{proof}
Since $s_i r_i=r_i s_i$ and $r_i=\pd_i s_i$, it follows that $s_i\pd_i=\pd_is_i$. Then, since $\pd_i^2=0$, $r_i^2=0$ as well. Due to Equation \ref{s_ir_j} from Lemma \ref{sr}, we have that $s_ir_j=r_js_i$ for $|i-j| \geq 2$, so $s_i\pd_j=\pd_js_i$. Thus, $r_i$ and $r_j$ anti-commute since $\pd_i\pd_j+\pd_j\pd_i=0$. The operators $r_i$ also satisfy braid relations, which we show by inductively reducing to $i=1$, and then using \ref{ekl lemma} and $s_i\pd_i=\pd_is_i$ repeatedly:
\begin{align*}
&r_1r_2r_1=s_1\pd_1 s_2\pd_2s_1\pd_1=s_1s_2\pd_{1,3}s_1\pd_{1,3}\pd_1=s_1s_2s_1\pd_{2,3}\pd_{1,3}\pd_{1,2} \\
&r_2r_1r_2=s_2\pd_2s_1\pd_1s_2\pd_2=s_2s_1\pd_{1,3}s_2\pd_{1,3}\pd_2=s_2s_1s_2\pd_{1,2}\pd_{1,3}\pd_{2,3}.
\end{align*}
Since $s_1s_2s_1=s_2s_1s_2$ and $\pd_{2,3}\pd_{1,3}\pd_{1,2}=\pd_{1,2}\pd_{1,3}\pd_{2,3}$ by symmetry, we conclude that $r_1r_2r_1=r_2r_1r_2$.
The Leibniz rule for $r_{i,k}$ (equation 4 of this lemma) follows immediately from the Leibniz rule for $\pd_{i,k}$. Since $r_i(x_i)=r_i(x_{i+1})=1$ and $r_i(x_j)=0$ for $j \neq i,i+1$, equations \ref{leib1} and \ref{leib3} follow from the Leibniz rule for $r_{i,k}$. 
\end{proof}
We also desire an explicit definition of the $r_{i,k}$ analogous to that of the odd divided difference operator of \cite{ekl}. To find such an expression, we use a preparatory lemma.
\begin{lem} \label{tau trans lem} For all $f \in $P$^-$ and $1 \leq i \neq k \leq n$, we have
\begin{equation*}
s_{i,k}x_i\tau_i(f)-s_{i,k}x_{k}\tau_{k}(f)=(-1)^{|f|}s_{i,k}(f)(x_i-x_k). 
\end{equation*}
\end{lem}
\begin{proof}
It suffices to prove the result for a monomial $x^{\lambda}=x_1^{\lambda_1}\ldots x_i^{\lambda_i} \ldots x_k^{\lambda_k} \ldots x_n^{\lambda_n}$, where $i<k$. We calculate that
\begin{align*}
&s_{i,k}x_i\tau_i(x^{\lambda})=(-1)^{\lambda_1+\ldots+\lambda_i}x_1^{\lambda_1}\ldots x_k^{\lambda_i+1} \ldots x_i^{\lambda_k} \ldots x_n^{\lambda_n} \\
&s_{i,k}x_{k}\tau_{k}(x^{\lambda})=(-1)^{\lambda_1+\ldots+\lambda_{k-1}}x_1^{\lambda_1}\ldots x_k^{\lambda_i} \ldots x_i^{\lambda_k+1} \ldots x_n^{\lambda_n} \\
&s_{i,k}(x^{\lambda})x_i=(-1)^{\lambda_{k+1}+\ldots+\lambda_n}x_1^{\lambda_1}\ldots x_k^{\lambda_i} \ldots x_i^{\lambda_{k+1}} \ldots x_n^{\lambda_n} \\
&s_{i,k}(x^{\lambda})x_k=(-1)^{\lambda_{i+1}+\ldots+\lambda_n}x_1^{\lambda_1}\ldots x_k^{\lambda_{i+1}} \ldots x_i^{\lambda_k} \ldots x_n^{\lambda_n}.
\end{align*}
Since $|f|=\lambda_1+\ldots+\lambda_n$, the desired result follows.
\end{proof} 
\begin{lem} \label{grand finale r_i} The operator $r_{i,k}$ has explicit form $r_{i,k}=(x_i^2-x_k^2)^{-1}[(x_i-x_k)s_{i,k}-x_i\tau_i+x_k\tau_k]$. \end{lem}
\begin{proof}
Follows from Lemma \ref{tau trans lem} and Equation \ref{kkoform}.
\end{proof}
We will now connect the above results to the odd Dunkl operator introduced by Khongsap and Wang in \cite{kw}.
\begin{defn} Define an operator $\delta_i$ by $\delta_i=(2x_i)^{-1}(1-\tau_i)$. \end{defn}
The above super-derivative can also be defined inductively, by imposing that $\delta_i(x_j)=1$ if $i=j$ and $0$ otherwise. We then extend the action to monomials as follows: 
\begin{equation*}
\delta_i(x_{a_1}x_{a_2}\ldots x_{a_{\ell}})=\sum_{k=1}^{\ell} (-1)^{k-1}x_{a_1}\ldots\delta_i(x_{a_k})x_{a_{k+1}}\ldots x_{a_{\ell}}.
\end{equation*}
The operator $\delta_i$ is \emph{a priori} from Laurent skew polynomials to Laurent skew polynomials, but it is easy to check that it preserves the subalgebra of skew polynomials.
Khongsap and Wang found an odd analog of the Dunkl operator, given by 
\begin{equation} \label{kw form}
\eta_i=t\delta_i+u \sum_{k \neq i} (x_i^2-x_k^2)^{-1}[(x_i-x_k)s_{i,k}-x_i\tau_i+x_k\tau_k],
\end{equation}
where $t,u \in \C^\times$. Their operators anti-commute; $\eta_i\eta_j+\eta_j\eta_i=0$ for $i \neq j$. 

By Lemma \ref{grand finale r_i}, this odd Dunkl operator may be expressed as 
\begin{equation} \label{etaexplicit}
\eta_i=t\delta_i+u \sum_{k \neq i} \pd_{i,k}s_{i,k}.
\end{equation}
By analogy with the commutative case, discussed in Section \ref{background}, the operator $r_{i,k}$ plays the same role in the odd theory that the even divided difference operator plays in the even theory. 

\section{Classical Yang-Baxter Equation and the Dunkl Laplacian} \label{cybe section}
\begin{thm} \label{yang baxter} Let 
\begin{equation}
H_{1,2,3}=[r_{1,2},r_{1,3}]_++[r_{1,3},r_{2,3}]_++[r_{1,2},r_{2,3}]_+,
\end{equation}
where $[p,q]_+=pq+qp$ is the anti-commutator. Then, the operators $r_{i,k}$ satisfy the classical Yang-Baxter equation
\begin{equation}
H_{1,2,3}=0 \label{yb}.
\end{equation}
\end{thm}
\begin{proof} To avoid a cumbersome direct calculation, we instead use an inductive approach. Namely, suppose that $H_{1,2,3}(f)=0$ for some function $f \in P^-$. Then, we show that $H_{1,2,3}(x_if)=0$, for all integers $i \geq 1$. Note that, for $i=1$, 
\begin{equation} \label{yb1}
\begin{aligned}
r_{1,2}r_{1,3}x_1&=r_{1,2}(s_{1,3}-x_1r_{1,3}) \\
&=r_{1,2}s_{1,3}-r_{1,2}(x_1r_{1,3}) \\
&=r_{1,2}s_{1,3}-(s_{1,2}r_{1,3}-x_1r_{1,2}r_{1,3}) \\
&=r_{1,2}s_{1,3}-s_{1,2}r_{1,3}+x_1r_{1,2}r_{1,3},
\end{aligned}
\end{equation}
where we have used Equation \ref{leibniz r_i} twice. Similarly, we find that
\begin{align}
&r_{1,3}r_{1,2}x_1=r_{1,3}s_{1,2}-s_{1,3}r_{1,2}+x_1r_{1,3}r_{1,2} \\
&r_{2,3}r_{1,2}x_1=r_{2,3}s_{1,2}+x_1r_{2,3}r_{1,2} \\
&r_{1,2}r_{2,3}x_1=-s_{1,2}r_{2,3}+x_1r_{1,2}r_{2,3} \\
&r_{1,3}r_{2,3}x_1=-s_{1,3}r_{2,3}+x_1r_{1,3}r_{2,3} \\
&r_{2,3}r_{1,3}x_1=r_{2,3}s_{1,3}+x_1r_{2,3}r_{1,3}. \label{yb6}
\end{align}
By our inductive hypothesis, 
\begin{equation*}
x_1(r_{1,2}r_{1,3}+r_{1,3}r_{1,2}+r_{2,3}r_{1,2}+r_{1,2}r_{2,3}+r_{1,3}r_{2,3}+r_{2,3}r_{1,3})=0. 
\end{equation*}
Keeping this in mind, add Equations \ref{yb1}-\ref{yb6} to show that
\begin{align*}
H_{1,2,3}x_1&=r_{1,2}s_{1,3}-s_{1,2}r_{1,3}+r_{1,3}s_{1,2}-s_{1,3}r_{1,2}+r_{2,3}s_{1,2}-s_{1,2}r_{2,3}-s_{1,3}r_{2,3}+r_{2,3}s_{1,3}\\
&=r_{1,2}s_{1,3}-r_{2,3}s_{1,2}+r_{1,3}s_{1,2}-r_{2,3}s_{1,3}+r_{2,3}s_{1,2}-r_{1,3}s_{1,2}-r_{1,2}s_{1,3}+r_{2,3}s_{1,3} \\
&=0,
\end{align*}
where we have repeatedly used Lemma \ref{sr} to slide $r_{i,j}$ past $s_{k,\ell}$. 

We can similarly show that $H_{1,2,3}x_2=H_{1,2,3}x_3=0$. Since $r_{j,k}x_i=x_ir_{j,k}$ for $i>3$ and $j,k \in (1,2,3)$, it also follows that $H_{1,2,3}x_i=0$ for $i>3$, proving the desired result. 
\end{proof}

\begin{cor}  \label{corollary combo}
The double summation $\sum_{i = 1} ^n \left(\sum_{k \neq i} r_{i,k}\right)^2=0$.
\end{cor}
\begin{proof}
The expansion of this double summation has $n(n-1)^2$ total terms. Since $r_{i,j}^2=0$, $n(n-1)$ of these terms are immediately zero, leaving $n(n-1)(n-2)$ terms of the form $r_{i,j}r_{k,\ell}$, where $i=k$ and $j=\ell$ are not both true. By Theorem \ref{yang baxter}, the sum of all six terms of the form $r_{i,j}r_{k,\ell}$, where $i,j,k,\ell \in {a,b,c}$ for distinct integers $1 \leq a,b,c\leq n$, $i \neq j$, and $k\neq \ell$, is zero. There are $\binom{n}{3}$ ways to choose integers $a,b,c$, and for each choice of $a,b,c$, six terms of the form $r_{i,j}r_{k,\ell}$ vanish. This fact eliminates all the remaining $6\binom{n}{3}=n(n-1)(n-2)$ terms of the double summation. 
\end{proof}

As an application of the results in this section, we will compute the odd Dunkl Laplacian: $\sum_{i=1}^n \eta_i^2$. We will first require a lemma involving the commutator of $\tau_i$ and $r_{i,k}$. 
\begin{lem} \label{prep laplacian}
The equation $x_i^{-1}[r_{i,k},\tau_i]=(x_i^2-x_k^2)^{-1}(s_{i,k}(\tau_i+\tau_k)-x_i^{-1}x_ks_{i,k}(\tau_i-\tau_k)-2)$ holds.
\end{lem}
\begin{proof}
By Lemma \ref{grand finale r_i}, 
\begin{align*}
&r_{i,k}\tau_i=(x_i^2-x_k^2)^{-1}((x_i-x_k)s_{i,k}\tau_i-x_i\tau_i^2+x_k\tau_k\tau_i) \\
&\tau_i r_{i,k}=(x_i^2-x_k^2)^{-1}(-(x_i+x_k)s_{i,k}\tau_k+x_i\tau_i^2+x_k\tau_i\tau_k).
\end{align*}
Since $\tau_i\tau_k=\tau_k\tau_i$ and $\tau_i^2=1$, the result follows by subtraction.
\end{proof}
Now, define $A_i=(2x_i)^{-1}(1-\tau_i)$ and $B_i=\sum_{ k \neq i}r_{i,k}$ so that the odd Dunkl operator $\eta_i$ of Khongsap and Wang may be expressed as $\eta_i=tA_i+uB_i$. Note that 
\begin{equation} \label{a_i^2}
\begin{aligned}
A_i^2&=\frac{1}{4}x_i^{-1}(1-\tau_i)x_i^{-1}(1-\tau_i) \\
&=\frac{1}{2}x_i^{-2}(1-\tau_i)(1-\tau_i) \\
&=x_i^{-2}(1-\tau_i),
\end{aligned}
\end{equation}
since $\tau_i(x_i^{-1})=-x_i^{-1}$ and $\tau_i^2=1$. 

\begin{lem} \label {sum ab+ba=0}
The relation $\sum_{i=1}^n (A_iB_i+B_iA_i)=0$ holds. 
\end{lem}
\begin{proof}
Due to the Leibniz Rule for $r_{i,k}$ (equation \ref{leibniz r_i}), we can find that $0=r_{i,k}(x_ix_i^{-1})=x_k^{-1}-x_ir_{i,k}(x_i^{-1})$, so $r_{i,k}(x_i^{-1})=x_i^{-1}x_k^{-1}$. It follows that $r_{i,k}(x_i^{-1}f)=x_i^{-1}x_k^{-1}s_{i,k}(f)-x_i^{-1}r_{i,k}(f)$. Using this fact, 
\begin{align*}
\sum_{i=1}^n B_iA_i&=\frac{1}{2}\sum_{i=1}^n\sum_{k \neq i} r_{i,k} x_i^{-1}(1-\tau_i) \\
&=\frac{1}{2}\sum_{i=1}^n\sum_{k \neq i} (x_i^{-1}x_k^{-1}s_{i,k}(1-\tau_i)-x_i^{-1}r_{i,k}(1-\tau_i)).
\end{align*}
By definition, 
\begin{equation*}
\sum_{i=1}^n A_iB_i=\frac{1}{2}\sum_{i=1}^n\sum_{k \neq i} x_i^{-1}(1-\tau_i)r_{i,k}.
\end{equation*}
Adding the above two equations, we find that
\begin{equation}  \label{ab+ba}
\begin{aligned}
\sum_{i=1}^n [A_i,B_i]_+&=\frac{1}{2}\sum_{i=1}^n\sum_{k \neq i} (x_i^{-1}x_k^{-1}s_{i,k}(1-\tau_i)+x_i^{-1}(r_{i,k}\tau_i-\tau_ir_{i,k})) \\
&=\frac{1}{2}\sum_{i=1}^n\sum_{k \neq i} (x_i^{-1}x_k^{-1}s_{i,k}(1-\tau_i)+(x_i^2-x_k^2)^{-1}(s_{i,k}(\tau_i+\tau_k)-x_i^{-1}x_ks_{i,k}(\tau_i-\tau_k)-2)),
\end{aligned}
\end{equation}
where we have used Lemma \ref{prep laplacian}.

Each double summation repeats the pair of indices $(p,q)$ twice when $1 \leq p,q \leq n$, one time when $i=p$ and $k=q$ and once more when $i=q$ and $k=p$. Note that
\begin{equation*}
(x_i^2-x_k^2)^{-1}s_{i,k}(\tau_i+\tau_k)=-(x_k^2-x_i^2)s_{k,i}(\tau_k+\tau_i).
\end{equation*}
As a result, the sum $\frac{1}{2}\sum_{i=1}^n\sum_{k \neq i}(x_i^2-x_k^2)^{-1}s_{i,k}(\tau_i+\tau_k)=0$. Similarly, 
\begin{equation*}
\frac{1}{2}\sum_{i=1}^n\sum_{k \neq i}x_i^{-1}x_k^{-1}s_{i,k}=0 \text{ and } \frac{1}{2}\sum_{i=1}^n\sum_{k \neq i}(x_i^2-x_k^2)^{-1}(-2)=0,
\end{equation*}
since $x_i^{-1}x_k^{-1}+x_k^{-1}x_i^{-1}=0$. Equation \ref{ab+ba} then becomes
\begin{equation} \label{ab+ba 2}
\sum_{i=1}^n (A_iB_i+B_iA_i)=\frac{1}{2}\sum_{i=1}^n\sum_{k \neq i} (-x_i^{-1}x_k^{-1}s_{i,k}\tau_i-(x_i^2-x_k^2)^{-1}(x_i^{-1}x_ks_{i,k}(\tau_i-\tau_k)).
\end{equation}
However, note that
\begin{equation*}
x_i^{-1}x_k+x_k^{-1}x_i=-(x_i^2-x_k^2)x_i^{-1}x_k^{-1},
\end{equation*}
which implies 
\begin{equation} \label{cancel1}
(x_i^2-x_k^2)^{-1}x_i^{-1}x_ks_{i,k}(\tau_i-\tau_k)+(x_k^2-x_i^2)^{-1}x_k^{-1}x_is_{k,i}(\tau_k-\tau_i)=-x_i^{-1}x_k^{-1}s_{i,k}(\tau_i-\tau_k).
\end{equation}
Similarly, we find that
\begin{equation}
x_i^{-1}x_k^{-1}s_{i,k}\tau_i+x_k^{-1}x_i^{-1}s_{i,k}\tau_k=x_i^{-1}x_k^{-1}s_{i,k}(\tau_i-\tau_k).
\end{equation}
As a result, equation \ref{ab+ba 2} becomes
\begin{equation*} 
\sum_{i=1}^n (A_iB_i+B_iA_i)=-\frac{1}{2}\sum_{1 \leq i < k \leq n} (x_i^{-1}x_k^{-1}s_{i,k}(\tau_i-\tau_k)-x_i^{-1}s_k^{-1}s_{i,k}(\tau_i-\tau_k))=0.
\end{equation*}
\end{proof}

We are now equipped to compute the Dunkl Laplacian in the odd case.

\begin{thm} 
The equation $\sum_{i=1}^n \eta_i^2=t^2\sum_{1 \leq i  \leq n}x_i^{-2}(1-\tau_i)$ holds.
\end{thm}
\begin{proof}
Since $\eta_i=tA_i+uB_i$, we have that 
\begin{equation*}
\sum_{i=1}^n \eta_i^2=t^2\sum_{i=1}^n A_i^2+tu\sum_{i=1}^n (A_iB_i+B_iA_i)+u^2\sum_{i=1}^n B_i^2
\end{equation*}
By Lemma \ref{sum ab+ba=0}, $\sum_{i=1}^n (A_iB_i+B_iA_i)=0$. By Corollary \ref{corollary combo}, $\sum_{i=1}^n B_i^2=0$. Therefore, 
\begin{equation*}
\sum_{i=1}^n \eta_i^2=t^2\sum_{i=1}^n A_i^2=t^2\sum_{1 \leq i  \leq n}x_i^{-2}(1-\tau_i),
\end{equation*}
by equation \ref{a_i^2}.
\end{proof}

\section{A Variant of the Khongsap-Wang Odd Dunkl Operator} \label{sl2}

In this section, we will show that a close variant of the odd Dunkl operator introduced by Khongsap and Wang can be used in the construction of three operators that satisfy the defining relations of the Lie algebra $\sltwo$. First, we will consider an operator $p_i$, which is different from $\delta_i$ but plays a similar role.
\begin{defn} The operator $p_i$ is a $\C$-linear map $P^- \rightarrow P^-$, which acts on monomials as follows: 
\begin{equation*}
p_i(x_1^{\lambda_1}\ldots x_i^{\lambda_i} \ldots x_n^{\lambda_n})=\lambda_i(-1)^{\lambda_1+\ldots+\lambda_{i-1}}x_1^{\lambda_1}\ldots x_i^{\lambda_i-1} \ldots x_n^{\lambda_n}. 
\end{equation*}
\end{defn}

\begin{rem} 
One may also introduce $p_i$ by using a nice Leibniz-like expression involving $\tau_i$, by defining
\begin{align*}
&p_i(x_j)=\delta_{ij}  \\
&p_i(fg)=p_i(f)g+(-1)^{|f}\tau_i(f)p_i(g) ,
\end{align*}
where $f,g \in P^-$ and $\delta_{ij}$ is the Kronecker delta. Now, note the analogous relationship between the degree-preserving operators $s_i$ and $\tau_i$ in their respective Leibniz rules for the $(-1)$-degree operators $\pd_i$ and $p_i$. This provides motivation for the definition of $p_i$ and suggests its natural role in our theory. 
\end{rem}

Now consider a modified version of $\eta_i$.
\begin{defn} Let \begin{equation} D_i=tp_i+u\sum_{k \neq i} r_{i,k}. \end{equation}
\end{defn}
\begin{defn} Introduce the odd $r^2$, Euler, and $\Delta$ operators as below:
\begin{align} \label{defops}
&r^2=(2t)^{-1}\sum_{i=1}^n x_i^2 \\
&E=\sum_{i=1}^n x_ip_i+\frac{n}{2}+\frac{u}{t}\sum_{k \neq i}s_{i,k} \\
&\Delta=-(2t)^{-1}\sum_{i=1}^n D_i^2.
\end{align}
\end{defn}
\begin{rem}
Heckman, who used the even Dunkl operators to find a $\sltwo$-triple useful in harmonic analysis, uses the convention $t=1$ \cite{heckman}. For now, we will consider $t$ to be a fixed constant in $\C^\times$. 
\end{rem}
\begin{rem}
The commutator in the setting of superalgebras is usually defined as $[a,b]=ab-(-1)^{|a||b|}ba$, where $|a|$ and $|b|$ are the degrees of $a$ and $b$, respectively. However, since all of the operators we will be considering in this section have even degree, there is no need to distinguish between commutators and super-commutators. 
\end{rem}

To construct an $\sltwo$ action from these operators, we will require a series of lemmas regarding the action of portions of the odd Euler operator $E$. In the next lemma, we investigate the action of the first term of the odd Euler operator on skew polynomials. 
\begin{lem} \label{lem1}The operator $\sum_{i=1}^n x_i p_i$ acts by multiplication by $|f|$ on the space of homogenous functions $f \in $P$^-$. \end{lem}
\begin{proof} It suffices to show the result for a monomial $x^{\lambda}=x_1^{\lambda_1}\ldots x_i^{\lambda_i} \ldots x_n^{\lambda_n}$. Note that 
\begin{equation*}
x_i p_i(x^{\lambda})=\lambda_ix_i(-1)^{\lambda_1+\ldots+\lambda_{i-1}}x_1^{\lambda_1}\ldots x_i^{\lambda_i-1} \ldots x_n^{\lambda_n}=\lambda_i x^{\lambda}.
\end{equation*}
By summing over all indices $i$, we obtain that 
\begin{equation*}
\sum_{i=1}^n x_i p_i(x^{\lambda})=(\lambda_1+\lambda_2+\ldots+\lambda_n)x^{\lambda},
\end{equation*}
which implies the desired result. 
\end{proof}
The above lemma holds true in the even case as well, where $p_i$ is replaced by the partial derivative with respect to $x_i$. We now prove some properties about the action of the third term of the odd Euler operator on $r^2$ and $\Delta$. 
\begin{lem} \label{lem2}  The commutation relation $\left[\sum_{k \neq i}^{} s_{i,k},\Delta\right]=0$ holds. \end{lem}
\begin{proof}
Note that $s_{j,k} p_i=p_j s_{j,k} \text{ if }i=k$, $s_{j,k} p_i=p_k s_{j,k} \text{ if }i=j$, and $s_{j,k} p_i=p_i s_{j,k} \text{ otherwise}$. Indeed, these relations can be verified by checking if they are true for $x_i^ax_j^bx_k^c$, $a,b,c \in \Z_+$, and then extending by linearity. We prove that $s_{j,k}p_i=p_j s_{j,k}$ if $i=k$, and the other two cases are similar. Without loss of generality, let $j<k$, and observe that
\begin{align*} 
&s_{j,k} p_k (x_j^ax_k^b)=b(-1)^as_{j,k}(x_j^{a}x_k^{b-1})= b(-1)^ax_k^ax_j^{b-1}=b(-1)^{ab}x_j^{b-1}x_k^a  \\
&p_j s_{j,k} (x_j^ax_k^b)=(-1)^{ab}p_j(x_j^bx_k^a)=b(-1)^{ab}x_j^{b-1}x_k^a.
\end{align*}
By our work in Lemma \ref{sr}, one can deduce that $s_{j,k}r_{\ell,m}=r_{s_{j,k}(\ell,m)} s_{i,j}$. As a consequence, we find that $s_{j,k} D_i=D_{s_{j,k}(i)} s_{j,k}$. By an easy induction, we now have that $s_{j,k} \Delta=\Delta s_{j,k}$. Using the above equation multiple times proves the desired result. 
\end{proof}

\begin{lem} \label{lem3} The commutation relation $\left[\sum_{k \neq i}s_{i,k},r^2\right]=0$ holds. \end{lem}
\begin{proof}
Follows since $s_{j,k} x_j=x_k s_{j,k}$, $s_{j,k} x_k=x_j s_{j,k}$ and $s_{j,k} x_i = x_i s_{j,k}$ if $i\neq j,k$. 
\end{proof}
We are now ready to obtain two commutativity relations involving the odd Euler operator $E$.
\begin{thm} The odd Euler operator and $r^2$ satisfy the following commutation relations:
\begin{align}
&[E,r^2]=2r^2 \\
&[E,\Delta]=-2\Delta.
\end{align}
\end{thm}
\begin{proof}
Since $r^2$ has degree $2$ and $\Delta$ has degree $-2$, the theorem follows from Lemmas \ref{lem1}, \ref{lem2}, and \ref{lem3}. 
\end{proof}

We also need to investigate what the third commutativity relation $[r^2,\Delta]$ turns out to be. We will prove one lemma before doing so. 
\begin{lem} \label{xD+Dx} For $i=1$ to $n$, the equation $x_iD_i+D_ix_i=2tx_ip_i+t+u\sum_{k \neq i}s_{i,k}$ holds. \end{lem}
\begin{proof}
Recall that
\begin{equation*}
D_i=tp_i+u\sum_{k \neq i} (x_i^2-x_k^2)^{-1}[(x_i-x_k)s_{i,k}-x_i\tau_i+x_k\tau_k].
\end{equation*}
Therefore, since $p_ix_i=x_ip_i+1$,
\begin{align*}
D_ix_i=tx_ip_i+t+u\sum_{k \neq i} (x_i^2-x_k^2)^{-1}(x_ix_k-x_k^2)s_{i,k}+\sum_{k \neq i}(x_i^2-x_k^2)[x_i^2\tau_i-x_ix_k\tau_k] \\
x_iD_i=tx_ip_i+u\sum_{k \neq i}(x_i^2-x_k^2)^{-1}(x_i^2-x_ix_k)s_{i,k}+\sum_{k \neq i}(x_i^2-x_k^2)^{-1}[-x_i^2\tau_i+x_ix_k\tau_k].
\end{align*}
Adding, we obtain the desired result. 
\end{proof}

We now have the tools to find the third relation between $r^2$, $E$, and $\Delta$:
\begin{thm} The commutation relation $[r^2,\Delta]=E$ holds. \end{thm}
\begin{proof}
We will first find $[r^2,D_i]$. The derivative $p_i$, much like the partial derivative in the even case, satisfies the properties $p_ix_j=-x_jp_i \text{ for } i \neq j$ and $p_ix_i=x_ip_i+1$. Now, suppose that $i \neq j$. Then, 
\begin{align*}
D_ix_j^2&=tp_ix_j^2+ux_j^2\sum_{k \neq i \neq j}r_{i,k}+x_k^2(x_i^2-x_k^2)^{-1}[-x_i\tau_i+x_k\tau_k]+x_i^2(x_i^2-x_k^2)^{-1}[(x_i-x_k)s_{i,k}] \\
&=tx_j^2p_i+ux_j^2\sum_{k \neq i}r_{i,k}+(x_i-x_j)s_{i,j} 
\end{align*}
Now, we will find $D_ix_i^2$: 
\begin{align*}
D_ix_i^2&=tp_ix_i^2+\sum_{k \neq i} x_k^2(x_i^2-x_k^2)^{-1}(x_i-x_k)s_{i,k}+x_i^2\sum_{k \neq i}(x_i^2-x_k^2)^{-1}[-x_i\tau_i+x_k\tau_k] \\
&=tx_i^2p_i+2tx_i+x_i^2\sum_{k \neq i}r_{i,k} - \sum_{k\neq i}(x_i-x_k)s_{i,k}.
\end{align*}
Therefore, $\left[\sum_{i=1}^nx_i^2,D_i\right]=-2tx_i$. This implies that
\begin{equation} \label{r^2D-Dr^2}
r^2D_i-D_ir^2=-x_i.
\end{equation}
As a result, we find that
\begin{align*}
[r^2,\Delta]&=-(2t)^{-1}\sum_{i=1}^n [r^2,D_j^2]=-(2t)^{-1}\sum_{i=1}^n(r^2D_j^2-D_j^2r^2)  \\
&=-(2t)^{-1}\sum_{i=1}^n[(D_jr^2D_j-x_jD_j)-(D_jr^2D_j+D_jx_j)]  \\
&=(2t)^{-1}\sum_{i=1}^n(x_iD_i+D_ix_i),
\end{align*}
where we have used \ref{r^2D-Dr^2}. Now, by Lemma \ref{xD+Dx}, 
\begin{equation*}
[r^2,\Delta]=\sum_{i=1}^nx_ip_i+\frac{n}{2}+\frac{u}{t}\sum_{k \neq i}s_{i,k}=E,
\end{equation*}
as desired. 
\end{proof}
To summarize, we have found operators $E$,$r^2$, and $\Delta$, similar to their even counterparts, which satisfy the defining relations of the Lie algebra $\sltwo$: 
\begin{align*}
&[E,r^2]=2r^2 \\
&[E,\Delta]=-2\Delta \\
&[r^2,\Delta]=E.
\end{align*}
\begin{rem}
If one uses the odd Dunkl operator $\eta_i$ as found in \cite{kw} instead of the $D_i$ introduced here, the $r^2$, $E$, and $\Delta$ operators do not generate $\sltwo$.
\end{rem}
\begin{rem}
Although our results hold true for all $t$ and $u$ in $\C$, one typically sets $t=1$ and $u=\alpha^{-1}$ for some $\alpha \in \C^{\times}$, since without loss of generality one of $t$ and $u$ may equal $1$. 
\end{rem}
\begin{rem}
In the even case, let $X$ be a Euclidean vector space with dimension $n$ and let $\C[X]$ be the algebra of $\C$-valued functions on $X$. Then, this result about $\sltwo$  plays a major role in the study of higher differential operators on $\C[X]$. This is because the representation theory of $\sltwo$ allows for the reduction of degree to the second order \cite{heckman}. As a result, our results in this section should correspondingly have a role in further studying differential operators in the odd case. 
\end{rem}

\section{$q$-nilHecke Algebras} \label{nh}

Until now, we have been concerned with the odd symmetric polynomials in variables $x_1,x_2,\ldots,x_n$, where $x_ix_j=(-1)x_jx_i$ for $1 \leq i \neq j \leq n$. This immediately suggests the question: what if one replaces the $-1$ by any constant $q \in \C^{\times}$? Specifically, we ask the following questions:
\begin{question} \label{q-question1}
Is it possible to study $q$-symmetric polynomials, for which $x_ix_j=qx_jx_i$ when $i>j$? 
\end{question}
\begin{question}
Are there $q$-analogs of even/odd divided difference operators and nilHecke algebras? So far, such structures are known only for the even case ($q=1$) and the odd case ($q=-1$).
\end{question}
In this section, we answer both questions in the affirmative.

We work in the $\Z$-graded, $q$-braided setting throughout.  Let $\C$ be a commutative ring and let $q\in\C^\times$ be a unit.  If $V,W$ are graded $\C$-modules and $v\in V$, $w\in W$ are homogeneous, the braiding is the ``$q$-twist'':
\begin{equation}\begin{split}
\tau_q:&V\otimes W\to W\otimes V\\
&v\otimes w\mapsto q^{|v||w|}w\otimes v,
\end{split}\end{equation}
where $|\cdot|$ is the degree function.  By $q$-algebra we mean an algebra object in the category of graded $\C$-modules equipped with this braided monoidal structure; likewise for $q$-bialgebras, $q$-Hopf algebras, and so forth.

\begin{rem}
Note that the $q$-twist described above is its own inverse only when $q^2=1$, which correlates to the even and odd cases. When $q^2 \neq 1$, the corresponding theory becomes more complex. Therefore, the $q$-nilHecke algebras that we introduce later in this section are nontrivial generalizations of the previously studied even and odd nilHecke algebras. 
\end{rem}

\begin{defn} The $q$-algebra $P^q_n$ is defined to be
\begin{equation}
P^q_n=\C\langle x_1,\ldots,x_n\rangle/(x_jx_i-qx_ix_j=0 \text{ if }i<j),
\end{equation}
where $|x_i|=1$ for $i=1,\ldots,n$.\end{defn}

Note that $P^q_n\cong\otimes_{i=1}^nP^q_1$.  There are two interesting subalgebras of $P^q_n$ that can be thought of as $q$-analogs of the symmetric polynomials.  Define the $k$-th \emph{elementary $q$-symmetric polynomial} to be
\begin{equation*}
e_k(x_1,\ldots,x_n)=\sum_{1\leq i_1<\ldots<i_k\leq n}x_{i_1}\cdots x_{i_n}
\end{equation*}
and define the $k$-th \emph{twisted elementary $q$-symmetric polynomial} to be
\begin{equation*}
\et_k(x_1,\ldots,x_n)=\sum_{1\leq i_1<\ldots<i_k\leq n}\xt_{i_1}\cdots\xt_{i_n},
\end{equation*}
where $\xt_j=q^{j-1}x_j$.

\begin{defn} The $q$-algebra of \emph{$q$-symmetric polynomials} in $n$ variables, denoted $\sym^q_n$, is the subalgebra of $P^q_n$ generated by $e_1,\ldots,e_n$.  Likewise for the \emph{twisted $q$-symmetric polynomials}, $\symt^q_n$, and $\et_1,\ldots,\et_n$.\end{defn}

The type A braid group on $n$ strands acts on $P^q_n$ by setting
\begin{multicols}{2}
\begin{enumerate}
\item $\sigma_i(x_j)=qx_{i+1} \text{ if } j=i$
\item $\sigma_i(x_j)=q^{-1}x_i \text{ if } j=i+1$
\item $\sigma_i(x_j)=qx_j \text{ if } j>i+1$
\item $\sigma_i(x_j)=q^{-1}x_j \text{ if } j<i$
\end{enumerate}
\end{multicols}
and extending multiplicatively. 
\begin{defn} For $i=1,\ldots,n-1$, the $i$-th \emph{$q$-divided difference operator} $\pd_i$ is the linear operator $P^q_n \to P^q_n$ defined by $\pd_i(x_i)=q$, $\pd_i(x_{i+1})=-1$, $\pd_i(x_j)=0$ for $j \neq i,i+1$, and 
\begin{equation} \label{leib}
\pd_i(fg)=\pd_i(f)g+\sigma_i(f)\pd_i(g),
\end{equation}
\end{defn}
for all functions $f,g \in P^q_n$. We call Equation \ref{leib} the $q$-Leibniz rule.  

\begin{lem} For every $i$ and every $j<k$, $\pd_i(x_kx_j-qx_jx_k)=0$.  \end{lem}
\begin{proof} 
Since $\pd_i(x_j)=0$ for $j>i+1$,  one may reduce the lemma to having to prove that $\pd_1(x_2x_1-qx_1x_2)=0$, $\pd_1(x_3x_1-qx_1x_3)=0$, and $\pd_1(x_3x_2-qx_2x_3)=0$. These statements follow from the $q$-Leibniz rule.  
\end{proof}
Therefore, $\pd_i$ is a well-defined operator on $P^q_n$.
\begin{lem} The following relations hold: \label{d_i x_i^k}
\begin{eqnarray*}
&\partial_i(x_i^k)=\sum\limits_{j=0}^{k-1}q^{jk-2j-j^2+k}x_i^jx_{i+1}^{k-1-j}\\
&\partial_i(x_{i+1}^k)=-\sum\limits_{j=0}^{k-1}q^{-j}x_i^jx_{i+1}^{k-1-j}.
\end{eqnarray*}
\end{lem}
\begin{proof}
We induct on $k$. The base case ($k=1$) follows from the definition of the $\pd_i$, and the powers of $q$ arise mostly from $x_{i+1}^n x_i^m=q^{mn}x_i^mx_{i+1}^n$ for all $m,n \in \Z_+$. 
\end{proof}

Our $q$-divided difference operators also annihilate the twisted elementary $q$-symmetric polynomials, just as the even divided difference operators annihilate the elementary symmetric functions.
\begin{lem} For every $i=1,\ldots,n-1$ and every $k$, $\pd_i(\et_k)=0$.
Hence $\symt^q_n\subseteq\bigcap_{i=1}^{n-1}\ker(\pd_i)$.\end{lem}
\begin{proof} We can express $\et_k$ as
\begin{equation*}
e_k=\sum_{\substack{|\undj|=k\\i,i+1\notin\undj}}\xt_{\undj}+\sum_{\substack{|\undj|=k-1\\i,i+1\notin\undj}}q^{f(\undj,i,k)}\xt_{\undj}(x_i+qx_{i+1})+\sum_{\substack{|\undj|=k-2\\i,i+1\notin\undj}}q^{g(\undj,i,k)}\xt_{\undj}x_ix_{i+1},
\end{equation*}
for certain $\Z$-valued functions $f,g$.  The result then follows from $\pd_i(x_i+qx_{i+1})=\pd_i(x_ix_{i+1})=0$ and the $q$-Leibniz rule.  
\end{proof}

Having discussed $q$-divided difference operators, we can now construct algebras for every $q \neq 0,1,-1$ that have many similarities to the even and odd nilHecke algebras. For every such $q$, we define a \emph{$q$-nilHecke algebra} generated by $x_i$ and $\pd_i$ for $1 \leq i \leq n$, subject to the relations found in the following two lemmas (\ref{qrels} and \ref{qbraid}).  

\begin{lem} \label{qrels} The following relations hold among the operators $\pd_i$ and $x_i$ (left multiplication by $x_i$):
\begin{multicols}{2}
\begin{enumerate}
\item $\pd_i^2=0 \label{delsquared}$
\item $\pd_j\pd_i-q\pd_i\pd_j=0\text{ for }j>i+1 \label{distance}$
\item $x_jx_i=qx_ix_j\text{ for }i<j \label{obvious}$
\item $\pd_ix_j-qx_j\pd_i=0\text{ for }j>i+1 \label{nilhecke1}$
\item $q\pd_ix_j-x_j\pd_i=0\text{ for }j<i \label{nilhecke2}$
\item $\pd_ix_i-qx_{i+1}\pd_i=q \label{akaleibniz1}$
\item $x_i\pd_i-q\pd_ix_{i+1}=q \label{akaleibniz2}$.
\end{enumerate}
\end{multicols}
\end{lem}
\begin{proof} 
To show that $\pd_i^2=0$, note that we can reduce to $i=1$ and proceed by induction. Since $\pd_i(1)=0$, the base case follows. Suppose that $\pd_i^2(f)=0$. Then, note that
\begin{align*}
&\pd_1^2(x_1f)=\pd_1(qf+qx_2\pd_1(f))= q\pd_1(f)-q\pd_1(f)+x_1\pd_1^2(f)=0  \\
&\pd_1^2(x_2f)=\pd_1(-f+q^{-1}x_1\pd_1(f))= -\pd_1(f)+\pd_1(f)+x_2\pd_1^2(f)=0  \\
&\pd_1^2(x_3f)=\pd_1(qx_3\pd_1(f))= q^2x_3\pd_1^2(f)=0,  
\end{align*}
which completes the proof of the first statement in the lemma. 

Statement \ref{obvious} follows by definition. Statements  \ref{nilhecke1}, \ref{nilhecke2}, \ref{akaleibniz1}, and \ref{akaleibniz2} follow from a suitable application of the $q$-Leibniz rule. Statement \ref{distance} follows from an inductive argument. We can reduce to $i=1$ and $j=3$. Suppose that $\pd_j\pd_i=q\pd_i\pd_j$ if $j>i+1$. Then,
\begin{align*}
&\pd_3\pd_1(x_1f)-q\pd_1\pd_3(x_1f)=(q\pd_3(f)+x_2\pd_3\pd_1(f))-q(\pd_3(f)+x_2\pd_1\pd_3(f))=0\\
&\pd_3\pd_1(x_2f)-q\pd_1\pd_3(x_2f)=(-\pd_3(f)+q^{-2}x_1\pd_3\pd_1(f))-q(-q^{-1}\pd_3(f)+q^{-2}x_1\pd_1\pd_3(f))=0. \\
&\pd_3\pd_1(x_3f)-q\pd_1\pd_3(x_3f)=(q^2\pd_1(f)+q^2x_4\pd_3\pd_1(f))-q(q\pd_1(f)+q^2x_4\pd_1\pd_3(f))=0 \\
&\pd_3\pd_1(x_4f)-q\pd_1\pd_3(x_4f)=(-q\pd_1(f)+x_3\pd_3\pd_1(f))-q(-\pd_1(f)+x_3\pd_1\pd_3(f))=0\\
&\pd_3\pd_1(x_5f)-q\pd_1\pd_3(x_5f)=q^2x_5\pd_3\pd_1(f)-q(q^2x_5\pd_1\pd_3(f))=0,
\end{align*}
thereby completing the induction.
\end{proof}

\begin{lem} \label{qbraid}
$\pd_i\pd_{i+1}\pd_i\pd_{i+1}\pd_i\pd_{i+1}+\pd_{i+1}\pd_i\pd_{i+1}\pd_i\pd_{i+1}\pd_i=0$.
\end{lem}
\begin{proof}
This result follows from an inductive argument; we reduce to $i=1$ and assume that the braid relation holds true for some function $f$. Then, we check that the braid relation is true for $x_1f$, $x_2f$, $x_3f$, and $x_4f$ (since the behavior of $x_jf$ for $j \geq 4$ is the same as that of $x_4f$). For brevity, we will show the argument for $x_2f$ only:
\\[6.5pt]

{\centering
    \begin{tabular}{p{7.8 cm}  p{7.8 cm}}
    $\pd_1\pd_2(x_2f)=q\pd_1(f)+q^2x_3\pd_1\pd_2(f)$ & $\pd_2\pd_1(x_2f)=-\pd_2(f)+q^{-2}x_1\pd_2\pd_1(f)$ \\[6.5pt]
    $\pd_{212}(x_2f)=q\pd_2\pd_1(f)-q^2\pd_1\pd_2(f)+qx_2\pd_{212}(f) $ &  $q\pd_{121}(x_2f)=-q\pd_1\pd_2(f)+\pd_2\pd_1(f)+\pd_{121}(f)$ \\[6.5pt] 
    $\pd_{1212}(x_2f)=q\pd_{121}(f)-q\pd_{212}(f)+x_1\pd_{1212}(f)$ & $\pd_{2121}(x_2f)=-\pd_{212}(f)+\pd_{121}(f)+x_3\pd_{2121}(f)$ \\[6.5pt] 
    $\pd_{21212}(x_2f)=q\pd_{2121}(f)+q^{-1}x_2\pd_{21212}(f)$ & $\pd_{12121}(x_2f)=-\pd_{1212}(f)+qx_3\pd_{12121}(f).$ \\\\
    \end{tabular}
\par} 
We continue the above calculations to find that
\begin{align*}
&\pd_{121212}(x_2f)=q\pd_{12121}(f)+\pd_{21212}(f)+x_2\pd_{121212}(f) \\
&\pd_{212121}(x_2f)=-q\pd_{12121}(f)-\pd_{21212}(f)+x_2\pd_{121212}(f), 
\end{align*}


and the braid relation for $x_2f$ follows from the inductive hypothesis. 
\end{proof}

\section{A Diagrammatic Approach to $q$-Symmetric Polynomials} \label{diagrams}

\subsection{Introduction to a $q$-Bialgebra}
In the previous section, we answered Question \ref{q-question1} in an algebraic way by defining $q$-analogs of the classical elementary and complete symmetric functions. In this section, we generalize the diagrammatic method used in \cite{ek} in order to study this question from the perspective of bialgebras.

Let $\nsym^q$ be a free, associative, $\Z$-graded $\C$-algebra with generators $h_m$ for $m \geq 0$. We define $h_0=1$ and $h_m=0$ for $m<0$, and let $q \in \C^{\times}$. The homogenous part of $\nsym^q$ of degree $\ell$ has a basis 
$\lbrace h_\alpha\rbrace_{\alpha\vDash k}$, where
\begin{equation*}
h_\alpha=h_{\alpha_1}\cdots h_{\alpha_z}\text{ for a composition }\alpha=(\alpha_1,\ldots,\alpha_z)\text{ of } \ell.
\end{equation*}
Define a multiplication for homogenous $x$ and $y$ on $\nsym^{q \otimes2}$ as follows, where deg($x$) denotes the degree of $x$: 
\begin{equation*}
(w \otimes x)(y \otimes z)=q^{\text{deg}(x)\text{deg}(y)}(wy \otimes xz).
\end{equation*}
We can make $\nsym^q$ into a $q$-bialgebra by letting the comultiplication on generators be
\begin{equation*}
\Delta(h_n)=\sum\limits_{k=0}^n h_k \otimes h_{n-k},
\end{equation*}
and by letting the counit be $\epsilon(x)=0$ if $x$ is homogenous and deg($x$)$>0$. 

We can impose, through the braiding structure, that: 
\begin{equation*}
\Delta(h_ah_b)=\sum\limits_{j=0}^{a}\sum\limits_{k=0}^b(h_j \otimes h_{a-j})(h_k \otimes h_{b-k})=\sum\limits_{j=0}^{a}\sum\limits_{k=0}^bq^{k(a-j)}(h_jh_k \otimes h_{a-j}h_{b-k}).
\end{equation*}
For any partitions $\lambda$ and $\mu$ of $n$, consider the set of double cosets of subroups $S_\lambda$ and $S_\mu$ of $S_n$: $S_\lambda\backslash S_n/S_\mu$. For every $C$ in this set, let $w_C$ be the minimal length representative of $C$ and let $\ell(w_C)$ be the length of this minimal length representative. We will now attribute a bilinear form to $\nsym^q$: 
\begin{equation*}
(h_\lambda,h_\mu)=\sum_{C\in S_\lambda\backslash S_n/S_\mu}q^{\ell(w_C)}.
\end{equation*}

This bilinear form admits a diagrammatic description. Let $h_n$ be an orange platform with $n$ non-intersecting strands coming out of it. When computing $(h_\lambda, h_\mu)$, with $\ell(\lambda)=z$ and $\ell(\mu)=y$, draw $z$ orange platforms at the top of the diagram, representing $\lambda_1$, $\lambda_2$,$\cdots$,$\lambda_z$. Draw $y$ orange platforms at the bottom of the diagram, representative of $\mu_1$, $\mu_2$,$\cdots$, $\mu_y$. We require that $|\lambda|=|\mu|$, so that the top platforms and bottom platforms have the same number of strands. 

Consider the example $(h_{121},h_{22})$. In the following diagram, snippets of the strands from each platform are shown. 
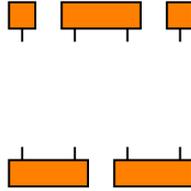
\begin{equation*}
\begin{tikzpicture}[scale=.7]
    \draw[thick] (1,0) [out=90, in=-90] to (1,.5);
    \draw[thick] (2,0) [out=90, in=-90] to (2,.5);
	\draw[thick] (3,0) [out=90, in=-90] to (3,.5);
	\draw[thick] (4,0) [out=90, in=-90] to (4,.5);
	\draw[thick] (1,2.5) [out=90, in=-90] to (1,3);
	\draw[thick] (2,2.5) [out=90, in=-90] to (2,3);
	\draw[thick] (3,2.5) [out=90, in=-90] to (3,3);
	\draw[thick] (4,2.5) [out=90, in=-90] to (4,3);
	\draw[thick, fill=orange] (.75,2.75) rectangle (1.25,3.25);
	\draw[thick, fill=orange] (1.75,2.75) rectangle (3.25,3.25);
	\draw[thick, fill=orange] (3.75,2.75) rectangle (4.25,3.25);
	\draw[thick, fill=orange] (.75,-.25) rectangle (2.25,.25);
	\draw[thick, fill=orange] (2.75,-.25) rectangle (4.25,.25);
\end{tikzpicture}
\end{equation*}

Every strand must start at one platform at the top and end on another platform at the bottom. No strands that have originated from one platform may intersect. The strands themselves have no critical points with respect to the height function, no two strands ever intersect more than once, and there are no triple-intersections where three strands are concurrent. Diagrams are considered up to isotopy. Without any restrictions, there would be $n!$ such diagrams if $|\lambda|=n$, since there would be no limitations on the ordering of the strands. However, due to the above rules, there are only $4$ possible diagrams in the computation of $(h_{121},h_{22})$, shown below. 

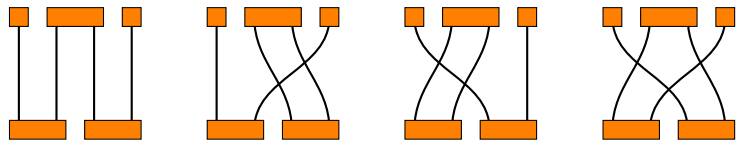
\begin{center}
\qquad
\begin{tikzpicture}[scale=.5]
	\draw[thick] (1,0) [out=90, in=-90] to (1,3);
	\draw[thick] (2,0) [out=90, in=-90] to (2,3);
	\draw[thick] (3,0) [out=90, in=-90] to (3,3);
	\draw[thick] (4,0) [out=90, in=-90] to (4,3);
	\filldraw[draw=black, fill=orange] (.75,2.75) rectangle (1.25,3.25);
	\filldraw[draw=black, fill=orange] (1.75,2.75) rectangle (3.25,3.25);
	\filldraw[draw=black, fill=orange] (3.75,2.75) rectangle (4.25,3.25);
	\filldraw[draw=black, fill=orange] (.75,-.25) rectangle (2.25,.25);
	\filldraw[draw=black, fill=orange] (2.75,-.25) rectangle (4.25,.25);
\end{tikzpicture}
\qquad
\begin{tikzpicture}[scale=.5]
	\draw[thick] (1,0) [out=90, in=-90] to (1,3);
	\draw[thick] (2,0) [out=90, in=-90] to (4,3);
	\draw[thick] (3,0) [out=90, in=-90] to (2,3);
	\draw[thick] (4,0) [out=90, in=-90] to (3,3);
	\filldraw[draw=black, fill=orange] (.75,2.75) rectangle (1.25,3.25);
	\filldraw[draw=black, fill=orange] (1.75,2.75) rectangle (3.25,3.25);
	\filldraw[draw=black, fill=orange] (3.75,2.75) rectangle (4.25,3.25);
	\filldraw[draw=black, fill=orange] (.75,-.25) rectangle (2.25,.25);
	\filldraw[draw=black, fill=orange] (2.75,-.25) rectangle (4.25,.25);
\end{tikzpicture}
\qquad
\begin{tikzpicture}[scale=.5]
	\draw[thick] (1,0) [out=90, in=-90] to (2,3);
	\draw[thick] (2,0) [out=90, in=-90] to (3,3);
	\draw[thick] (3,0) [out=90, in=-90] to (1,3);
	\draw[thick] (4,0) [out=90, in=-90] to (4,3);
	\filldraw[draw=black, fill=orange] (.75,2.75) rectangle (1.25,3.25);
	\filldraw[draw=black, fill=orange] (1.75,2.75) rectangle (3.25,3.25);
	\filldraw[draw=black, fill=orange] (3.75,2.75) rectangle (4.25,3.25);
	\filldraw[draw=black, fill=orange] (.75,-.25) rectangle (2.25,.25);
	\filldraw[draw=black, fill=orange] (2.75,-.25) rectangle (4.25,.25);
\end{tikzpicture}
\qquad
\begin{tikzpicture}[scale=.5]
	\draw[thick] (1,0) [out=90, in=-90] to (2,3);
	\draw[thick] (2,0) [out=90, in=-90] to (4,3);
	\draw[thick] (3,0) [out=90, in=-90] to (1,3);
	\draw[thick] (4,0) [out=90, in=-90] to (3,3);
	\filldraw[draw=black, fill=orange] (.75,2.75) rectangle (1.25,3.25);
	\filldraw[draw=black, fill=orange] (1.75,2.75) rectangle (3.25,3.25);
	\filldraw[draw=black, fill=orange] (3.75,2.75) rectangle (4.25,3.25);
	\filldraw[draw=black, fill=orange] (.75,-.25) rectangle (2.25,.25);
	\filldraw[draw=black, fill=orange] (2.75,-.25) rectangle (4.25,.25);
\end{tikzpicture}
\end{center}

Define
\begin{equation}
(h_\lambda,h_\mu)=\sum_{\text{all diagrams D representing } (h_\lambda,h_\mu)}q^{\text{ number of of crossings in }D}.
\end{equation}
In the above example, $(h_{121},h_{22})=1+2q^2+q^3$.

We can extend the bilinear form to $\nsym^{q\otimes 2}$ by stating that any diagram in which strands from distinct tensor factors intersect contributes $0$ to the bilinear form: 
\begin{equation*}
(w\otimes x,y\otimes z)=(w,y)(x,z).
\end{equation*}

Let $I$ be the radical of the bilinear form in $\nsym^q$. In \cite{ek}, the authors prove for any $q$ that multiplication and comultiplication are adjoint. In other words, for all $x$,$y_1$, $y_2$ in $\nsym^q$, 
\begin{equation} \label{adjoint}
(y_1 \otimes y_2, \Delta(x))=(y_1y_2,x).
\end{equation}

\subsection{The Elementary $q$-Symmetric Functions}

We now use the bilinear form of $q$-symmetric functions to study one of their important bases: the elementary $q$-symmetric functions. 

Define elements $e_k\in\nsym^q$ by $e_k=0$ for $k<0$, $e_0=1$, and
\begin{equation}\label{eqn-defn-e}
\sum_{i=0}^k(-1)^iq^{\binom{i}{2}}e_ih_{k-i}=0\text{ for }k\geq1.
\end{equation}
Equivalently, let
\begin{equation}
e_n=q^{-\binom{n}{2}}\sum_{\alpha\vDash n}(-1)^{\ell(\alpha)-n}h_\alpha.
\end{equation}

\begin{lem}

\begin{align*}
&1. \text{ The coproduct of an elementary function is given by } \Delta(e_n)=\sum_{k=0}^n e_k\otimes e_{n-k}.\\
&2. \text{ If } \lambda\vDash n,\text{ then } (h_\lambda,e_n)=\begin{cases}1&\text{if }\lambda=(1,\ldots,1)\\0&\text{otherwise.}\end{cases}
\end{align*}
\end{lem}

\begin{proof} 

We begin by demonstrating (2), from which (1) will follow. To show (2), it suffices to show that 
\begin{displaymath}
    (h_m x,e_n)=
    \left\{
    \begin{array}{lr}
       (x,e_{n-1}) & \text{ if } m=1 \\
       0 & \text{otherwise.}
     \end{array}
     \right.
\end{displaymath}
We will utilize strong induction on $n$ in order to find $(h_m x, e_k h_{n-k})$. The base cases $n=0,1$ are easy to show. 
There are two cases to consider by the inductive hypothesis applied to $k<n$. Either there is a strand connecting $h_m$ and $e_k$, or there is not. Just as we used an orange platform to denote $h_n$, we will use a blue platform to denote $e_k$. The rules of the diagrammatic notation are the same for the blue platforms as they are for the orange platforms. 
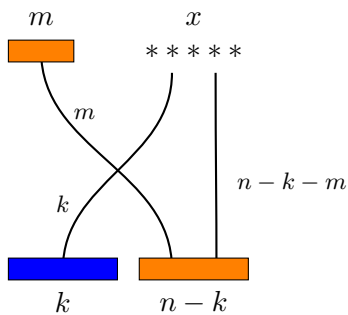
\qquad
\begin{center}
\begin{tikzpicture}[scale=.58] 
    \draw[thick] (2,0) [out=90, in=-90] to (4.5,4.5);
	\draw[thick] (5.52,-0.25) [out=90, in=-90] to (5.5,4.5);
	\draw[thick] (4.5,0) [out=90, in=-90] to (1.5,5);
	\node at (2,1.5) () {\footnotesize$k$};
	\node at (2.5,3.55) () {\footnotesize$m$};
	\node at (7.25,2) () {\footnotesize$n-k-m$};
	\filldraw[draw=black, fill=blue] (.75,-.25) rectangle (3.25,.25);
	\filldraw[draw=black, fill=orange] (3.75,-.25) rectangle (6.25,.25);
	\filldraw[draw=black, fill=orange] (.75,4.75) rectangle (2.25,5.25);
	\node () at (2,-.75) {$k$};
	\node () at (5,-.75) {$n-k$};
	\node () at (1.5,5.75) {$m$};
	\node () at (5,5) {$*****$};
	\node () at (5,5.75) {$x$};
\end{tikzpicture}
\end{center}
If there is not a strand connecting $h_m$ and $e_k$, the configuration contributes $q^{km}(x,e_k h_{n-k-m})$ .

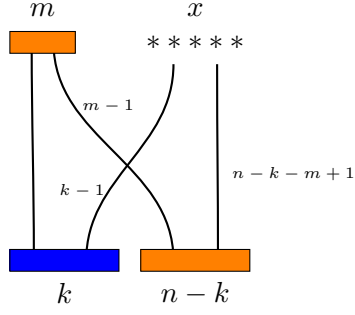
\qquad
\begin{center}
\begin{tikzpicture}[scale=.58]
	\draw[thick] (1.3,0) [out=90, in=-90] to (1.25, 5);
	\draw[thick] (2.5,0) [out=90, in=-90] to (4.5,4.5);
	\draw[thick] (5.52,0) [out=90, in=-90] to (5.5,4.5);
	\draw[thick] (4.5,0) [out=90, in=-90] to (1.75,5);
	\node at (2.4,1.6) () {\tiny$k-1$};
	\node at (3,3.55) () {\tiny$m-1$};
	\node at (7.25,2) () {\tiny$n-k-m+1$};
	\filldraw[draw=black, fill=blue] (.75,-.25) rectangle (3.25,.25);
	\filldraw[draw=black, fill=orange] (3.75,-.25) rectangle (6.25,.25);
	\filldraw[draw=black, fill=orange] (.75,4.75) rectangle (2.25,5.25);
	\node () at (2,-.75) {$k$};
	\node () at (5,-.75) {$n-k$};
	\node () at (1.5,5.75) {$m$};
	\node () at (5,5) {$*****$};
	\node () at (5,5.75) {$x$};
\end{tikzpicture}
\end{center}
If a stand connects $h_m$ and $e_k$, this configuration contributes $q^{(k-1)(m-1)}(x,e_{k-1}h_{n-k-m+1})$. 
We have thus shown that $(h_m x, e_k h_{n-k})=q^{km}(x,e_k h_{n-k-m})+q^{(k-1)(m-1)}(x,e_{k-1}h_{n-k-m+1})$.
Now we are equipped to consider $(h_m x, e_k)$.
\begin{align*}
(-1)^{n+1} q^{\binom{n}{2}}(h_m x, e_n)&=\sum\limits_{k=0}^{n-1}(-1)^k q^{\binom{k}{2}}(h_m x, e_k h_{n-k})\\
&=\sum\limits_{k=0}^{n-1}(-1)^k q^{\binom{k}{2}+km}(x,e_kh_{n-k-m})+\sum\limits_{k=0}^{n-1}(-1)^k q^{\binom{k}{2}+(m-1)(k-1)}(x,e_{k-1}h_{n-k-m+1})\\
&=\sum\limits_{k=0}^{n-1}(-1)^k q^{\binom{k}{2}+km}(x,e_kh_{n-k-m})+\sum\limits_{k=0}^{n-2}(-1)^{k+1} q^{\binom{k+1}{2}+(m-1)(k)}(x,e_{k}h_{n-k-m})\\
&=(-1)^{n-1}q^{\binom{n-1}{2}+nm}(x,e_{n-1}h_{1-m})
\end{align*}
Corresponding terms from the two sums cancel in pairs, since $q^{\binom{k}{2}+km}=q^{\binom{k+1}{2}+k(m-1)}$, leaving only the $k=n-1$ term in the first sum. The second statement of the lemma thus follows. 

We will now use (2) to prove (1). This follows from equation \ref{adjoint};  
\begin{equation*}
(\Delta(e_k),h_\lambda\otimes h_\mu)=(e_k,h_\lambda h_\mu)=\begin{cases}1
&\lambda=(1^\ell),\mu=(1^p),\ell+p=k,\\0&\text{otherwise.}\end{cases}
\end{equation*}
\end{proof}

We now calculate the sign incurred when strands connect two blue ($e_k$) platforms: 
\begin{align*}
(-1)^{n+1}q^{\binom{n}{2}}(e_n,e_n)&=\sum_{k=0}^{n-1}(-1)^kq^{\binom{k}{2}}(e_n,e_kh_{n-k})\\
&=(-1)^{n-1}q^{\binom{n-1}{2}}(e_n,e_{n-1}h_1)\\
&=(-1)^{n-1}q^{\binom{n-1}{2}}(\Delta(e_n),e_{n-1}\otimes h_1)\\
&=(-1)^{n-1}q^{\binom{n-1}{2}}\sum_{k=0}^{n}(e_k \otimes e_{n-k}, e_{n-1}\otimes h_1)\\
&=(-1)^{n-1}q^{\binom{n-1}{2}}(e_{n-1},e_{n-1}).
\end{align*}
One may solve this recursion to find that $(e_n, e_n)=q^{-\binom{n}{2}}$.Here, the second equality follows from noting that at most one strand can connect $h_{n-k}$ and $e_n$ (so that $k=n-1$), the third equality follows from adjointness, and the fourth and fifth equalities follow from the diagrammatic considerations of the previous lemma.  

To summarize the diagrammatics of the bilinear form thus developed:
\begin{enumerate}
\itemsep0em 
\item For each crossing, there is a factor of $q$ in the bilinear form. 
\item If two blue platforms are connected by $n$ strands, there is a factor of $q^{-\binom{n}{2}}$
\item At most one strand can connect a blue platform to an orange one. 
\end{enumerate}

\subsection{Relations Between Elementary $q$-Symmetric Polynomials}
In this subsection, we apply the diagrammatic method in order to study relations between $q$-elementary symmetric polynomials. 

Define $\text{Sym}^q\cong\nsym^q/R$, where $R$ is the radical of our bilinear form. 
\begin{lem}
If $q^n=1$, then $h_1^n$ is in the center of $\nsym^q$. 
\end{lem}
\begin{proof}
First, suppose $q$ is a primitive $n^{\text{th}}$ root of unity. Construct all ordered $k+1$-tuples of nonnegative integers that sum to $n-k$. Let $R_{k+1}^{n-k}$ be the set of all such $k+1$-tuples. For any tuple $(a_1,a_2, \cdots, a_{k+1})$, let $|(a_1,a_2, \cdots, a_{k+1})|$ be the sum of the entries of the tuple. 

For these tuples, $(a_1,a_2, \cdots, a_{k+1})$, define the map $f$ as follows: 
\begin{equation*}
f(a_1,a_2, \cdots, a_{k+1})=(ka_1,(k-1)a_2, (k-2)a_3,\cdots, a_k, 0).
\end{equation*}

Define 
\begin{equation*}
P(n,k)=\sum_{R_{k+1}^{n-k}} q^{|f(a_1,a_2, \cdots, a_{k+1})|}.
\end{equation*}

\begin{example}
\begin{equation*}
P(7,2)=1+q+2q^2+2q^3+3q^4+3q^5+3q^6+2q^7+2q^8+q^9+q^{10}
\end{equation*}
\end{example}

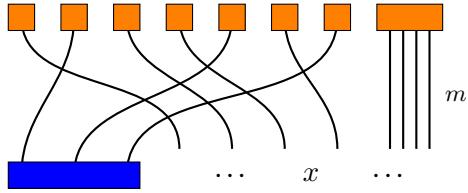
\begin{center}
\begin{tikzpicture}[scale=.7]
	\draw[thick] (1,0) [out=90, in=-90] to (2,3);
	\draw[thick] (2,0) [out=90, in=-90] to (5,3);
	\draw[thick] (3,0) [out=90, in=-90] to (7,3);
	\draw[thick] (4,.5) [out=90, in=-90] to (1,3);
	\draw[thick] (5,.5) [out=90, in=-90] to (3,3);
	\draw[thick] (6,.5) [out=90, in=-90] to (4,3);
	\draw[thick] (7,.5) [out=90, in=-90] to (6,3);
	\draw[thick] (8,.5) [out=90, in=-90] to (8,3);
	\draw[thick] (8.25,.5) [out=90, in=-90] to (8.25,3);
	\draw[thick] (8.5,.5) [out=90, in=-90] to (8.5,3);
	\draw[thick] (8.75,.5) [out=90, in=-90] to (8.75,3);
	\node at (9.25,1.5) {\footnotesize$m$};
	\node at (5,0) {\ldots};
	\node at (6.5,0) {$x$};
	\node at (8,0) {\ldots};
	\filldraw[draw=black, fill=blue] (.75,-.25) rectangle (3.25,.25);
	\filldraw[draw=black, fill=orange] (.75,2.75) rectangle (1.25,3.25);
	\filldraw[draw=black, fill=orange] (1.75,2.75) rectangle (2.25,3.25);
	\filldraw[draw=black, fill=orange] (2.75,2.75) rectangle (3.25,3.25);
	\filldraw[draw=black, fill=orange] (3.75,2.75) rectangle (4.25,3.25);
	\filldraw[draw=black, fill=orange] (4.75,2.75) rectangle (5.25,3.25);
	\filldraw[draw=black, fill=orange] (5.75,2.75) rectangle (6.25,3.25);
	\filldraw[draw=black, fill=orange] (6.75,2.75) rectangle (7.25,3.25);
	\filldraw[draw=black, fill=orange] (7.75,2.75) rectangle (9,3.25);
\end{tikzpicture}
\end{center}

Consider the above diagram, representative of $(h_1^nh_m, e_kx)$. In the diagram, $n=7$ and $m=3$. The three strands from $e_3$ "split" the seven $h_1$'s into groups of $1$, $2$, $1$, and $0$. This is a $3+1$-tuple that sums to $7-3=n-k=4$. Numbering the $h_1$'s from left to right, note that the first $h_1$ contributes $q^k$ intersections, the third and fourth $h_1$'s contribute $q^{k-1}$ intersections, and so on. In general, the diagrams in which no strand connects $h_m$ and $e_k$ contribute $P(n,k)(h_1^{n-k}h_m,x)$ to $(h_1^nh_m,e_kx)$.

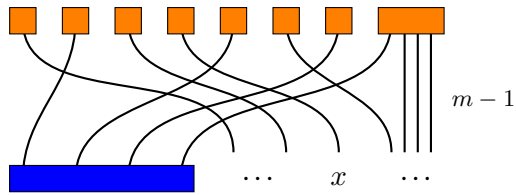
\begin{center}
\begin{tikzpicture}[scale=.7]
	\draw[thick] (1,0) [out=90, in=-90] to (2,3);
	\draw[thick] (2,0) [out=90, in=-90] to (5,3);
	\draw[thick] (3,0) [out=90, in=-90] to (7,3);
	\draw[thick] (4,0) [out=90, in=-90] to (8,3);
	\draw[thick] (5,.5) [out=90, in=-90] to (1,3);
	\draw[thick] (6,.5) [out=90, in=-90] to (3,3);
	\draw[thick] (7,.5) [out=90, in=-90] to (4,3);
	\draw[thick] (8,.5) [out=90, in=-90] to (6,3);
	\draw[thick] (8.25,.5) [out=90, in=-90] to (8.25,3);
	\draw[thick] (8.5,.5) [out=90, in=-90] to (8.5,3);
	\draw[thick] (8.75,.5) [out=90, in=-90] to (8.75,3);
	\node at (9.75,1.5) {\footnotesize$m-1$};
	\node at (5.5,0) {\ldots};
	\node at (7,0) {$x$};
	\node at (8.5,0) {\ldots};
	\filldraw[draw=black, fill=blue] (.75,-.25) rectangle (4.25,.25);
	\filldraw[draw=black, fill=orange] (.75,2.75) rectangle (1.25,3.25);
	\filldraw[draw=black, fill=orange] (1.75,2.75) rectangle (2.25,3.25);
	\filldraw[draw=black, fill=orange] (2.75,2.75) rectangle (3.25,3.25);
	\filldraw[draw=black, fill=orange] (3.75,2.75) rectangle (4.25,3.25);
	\filldraw[draw=black, fill=orange] (4.75,2.75) rectangle (5.25,3.25);
	\filldraw[draw=black, fill=orange] (5.75,2.75) rectangle (6.25,3.25);
	\filldraw[draw=black, fill=orange] (6.75,2.75) rectangle (7.25,3.25);
	\filldraw[draw=black, fill=orange] (7.75,2.75) rectangle (9,3.25);
\end{tikzpicture}
\end{center}

If a strand connects $e_k$ to $h_m$, then it intersects the other $n-(k-1)$ strands connecting some $h_1$ to $x$, contributing a factor of $q^{n-k+1}$.The other intersections contribute $P(n,k-1)$. Putting this case and the previous case together, we obtain that

\begin{eqnarray}\label{Case 1}
(h_1^nh_m,e_kx)=P(n,k)(h_1^{n-k}h_m,x)+q^{n-k+1}P(n,k-1)(h_1^{n-k+1}h_{m-1},x).
\end{eqnarray}

\begin{center}
\begin{tikzpicture}[scale=.7]
	\draw[thick] (1.25,0) [out=90, in=-90] to (4,3);
	\draw[thick] (2.25,0) [out=90, in=-90] to (6,3);
	\draw[thick] (3.25,.5) [out=90, in=-90] to (1.25,3);
	\draw[thick] (3.5,.5) [out=90, in=-90] to (1.5,3);
	\draw[thick] (3.75,.5) [out=90, in=-90] to (1.75,3);
	\draw[thick] (4,.5) [out=90, in=-90] to (2,3);
	\draw[thick] (5,.5) [out=90, in=-90] to (3,3);
	\draw[thick] (6,.5) [out=90, in=-90] to (5,3);
	\draw[thick] (7,.5) [out=90, in=-90] to (7,3);
	\draw[thick] (8,.5) [out=90, in=-90] to (8,3);
	\node at (1,2) {\footnotesize$m$};
	\node at (4,0) {\ldots};
	\node at (5.5,0) {$x$};
	\node at (7,0) {\ldots};
	\filldraw[draw=black, fill=blue] (1,-.25) rectangle (2.5,.25);
	\filldraw[draw=black, fill=orange] (1,2.75) rectangle (2.25,3.25);
	\filldraw[draw=black, fill=orange] (2.75,2.75) rectangle (3.25,3.25);
	\filldraw[draw=black, fill=orange] (3.75,2.75) rectangle (4.25,3.25);
	\filldraw[draw=black, fill=orange] (4.75,2.75) rectangle (5.25,3.25);
	\filldraw[draw=black, fill=orange] (5.75,2.75) rectangle (6.25,3.25);
	\filldraw[draw=black, fill=orange] (6.75,2.75) rectangle (7.25,3.25);
	\filldraw[draw=black, fill=orange] (7.75,2.75) rectangle (8.25,3.25);
\end{tikzpicture}

\begin{tikzpicture}[scale=.7]
	\draw[thick] (1,0) [out=90, in=-90] to (1.25,3);
	\draw[thick] (2,0) [out=90, in=-90] to (5,3);
	\draw[thick] (3,0) [out=90, in=-90] to (6,3);
	\draw[thick] (3.75,.5) [out=90, in=-90] to (1.5,3);
	\draw[thick] (4,.5) [out=90, in=-90] to (1.75,3);
	\draw[thick] (4.25,.5) [out=90, in=-90] to (2,3);
	\draw[thick] (5,.5) [out=90, in=-90] to (3,3);
	\draw[thick] (6,.5) [out=90, in=-90] to (4,3);
	\draw[thick] (7,.5) [out=90, in=-90] to (7,3);
	\draw[thick] (8,.5) [out=90, in=-90] to (8,3);
	\node at (2,1.5) {\tiny$m-1$};
	\node at (5,0) {\ldots};
	\node at (6.25,0) {$x$};
	\node at (7.5,0) {\ldots};
	\filldraw[draw=black, fill=blue] (.75,-.25) rectangle (3.25,.25);
	\filldraw[draw=black, fill=orange] (1,2.75) rectangle (2.25,3.25);
	\filldraw[draw=black, fill=orange] (2.75,2.75) rectangle (3.25,3.25);
	\filldraw[draw=black, fill=orange] (3.75,2.75) rectangle (4.25,3.25);
	\filldraw[draw=black, fill=orange] (4.75,2.75) rectangle (5.25,3.25);
	\filldraw[draw=black, fill=orange] (5.75,2.75) rectangle (6.25,3.25);
	\filldraw[draw=black, fill=orange] (6.75,2.75) rectangle (7.25,3.25);
	\filldraw[draw=black, fill=orange] (7.75,2.75) rectangle (8.25,3.25);
\end{tikzpicture}
\end{center}

Similarly, the above two diagrams show that 
\begin{eqnarray}\label{Case 2}
(h_mh_1^n,e_kx)=q^{mk}P(n,k)(h_mh_1^{n-k},x)+q^{(m-1)(k-1)}P(n,k-1)(h_{m-1}h_1^{n-k+1},x).
\end{eqnarray}

Now, consider the case when $k=n+1$. In this case, there is only one diagram for the bilinear form, and it can be shown that
\begin{equation*}
        \left\{
    \begin{array}{lr}
     (h_1^n h_m,e_{n+1}x)=(h_{m-1},x)&\\
     (h_m h_1^n,e_{n+1}x)=q^{n(m-1)}(h_{m-1},x),&\\
     \end{array}
     \right.
\end{equation*}
which are equal since $q^n=1$. Now, if $k \leq n$, we claim that $P(n,k)=0$ for all $n \neq k$. This follows from the fact that $q^n=1$, that $q^{n-\ell)}\neq1$ for $\ell \in (1,2,3,\cdots,n-1)$, and the fact that 
\begin{equation*}
P(n,k)=\binom{n}{k}_q.
\end{equation*}

The above statement follows from a bijection establishing $P(n,k)$ as the Gaussian binomial coefficient $\binom{n}{k}_q$. It is known that the coefficient of $q^j$ in $\binom{n}{k}_q$ is the number of partitions of $j$ into $k$ or fewer parts, with each part less than or equal to $k$. $P(n,k)$ yields the same result since $f$ takes every $k+1$-tuple to a $k+1$-tuple with last term $0$. Each term must be less than or equal to $n-k$ since we have imposed that the sum of all the terms is $n-k$. 

We substitute $P(n,k)=0$ in \eqref{Case 1} and \eqref{Case 2} to find that both products $(h_1^n h_m,e_kx)$ and $(h_m h_1^n, e_kx)$ are $0$ unless $n=k$ or $n=k-1$ (already addressed). If $n=k$, then 
\begin{equation*}
        \left\{
    \begin{array}{lr}
     (h_1^n h_m,e_{n}x)=(h_{m},x)+qP(n,n-1)(h_1h_{m-1},x)&\\
     (h_m h_1^n,e_{n}x)=q^{nm}(h_{m},x)+q^{(m-1)(n-1)}P(n,n-1)(h_{m-1}h_1,x).&\\
     \end{array}
     \right.
\end{equation*}
Since $q^{mn}=1$ and $P(n,n-1)=0$, the above two expressions are equal. We therefore have the desired result when $q$ is a primitive root of unity. By using some basic number theory and the recursive property of the Gaussian polynomials that 
\begin{equation*}
\binom{n}{k}_ q=q^k\binom{n-1}{k}_q+\binom{n-1}{k-1}_q,
\end{equation*}
one may extend the result to any root of unity. 
\end{proof}


\subsection{Insertion}

In this subsection, we develop the novel idea of insertion as a method for developing further relations in $\nsym^q$. Note from the previous arguments in this section that many diagrammatic relations between elementary symmetric functions involve evaluating the bilinear form $(h_{\lambda}, e_k x)$, for some $\lambda$, $k$, and $x \in \nsym^q$. The insertion method aids in the general computation of this bilinear form. 

Let $\lambda$ and $\mu$ be compositions such that $\lambda=(\lambda_1, \lambda_2,\cdots, \lambda_z)$ and $\mu=(\mu_1, \mu_2,\cdots, \mu_z)$. The length of $\lambda$ and $\mu$, which will be denoted by $\ell(\lambda)$ and $\ell(\mu)$, is $z$. Define $|\lambda|=\lambda_1+\lambda_2+\cdots +\lambda_z$. Let $\sigma_\ell^k$ be a binary sequence of $0$'s and $1$'s with $k$ total elements, $\ell$ of which are $1$. Let $O_\ell^k$ be the set of all $\sigma_\ell^k$ for given $k$ and $\ell$. The size of the set $O_\ell^k$ is $\binom{k}{\ell}$. 

Define subtraction and multiplication of compositions in a component-wise manner
\begin{equation*}
\lambda-\mu=(\lambda_1-\mu_1, \lambda_2-\mu_2,\cdots, \lambda_z-\mu_z)
\end{equation*}
\begin{equation*}
\lambda\mu=(\lambda_1\mu_1, \lambda_2\mu_2,\cdots, \lambda_z\mu_z).
\end{equation*}

Let $T_n^m$ be the composition with $m$ elements, all of which are $n$. Let $\lambda_k^G=(\lambda_{k+1}, \lambda_{k+2},\cdots, \lambda_z)$ and let $\lambda_k^L=(\lambda_1, \lambda_2,\cdots, \lambda_k)$. Further, let $r(\lambda)$ denote the composition $(\lambda_1, \lambda_1+\lambda_2, \lambda_1+\lambda_2+\lambda_3,\cdots, \lambda_1+\cdots+\lambda_z)$. 
\newline Define $(h_\lambda, e_k x)_{h_\mu}$ to be the result when computing $(h_\lambda,e_ kx)$, but with $h_\mu$
appended to the beginning of $h_\alpha$ all bilinear forms $(h_\alpha,x)$. We call this process insertion. 

\begin{example}
\begin{equation*}
(h_2h_3,e_1x)_{h_1}=(h_1h_1h_3,x)+q^2(h_1h_2h_2,x) 
\end{equation*}
\end{example}

We now show some applications of insertion. The first is a result that simplifies the computation of a specific bilinear form. 
\begin{lem}
The equation $(h_nh_\lambda,e_kx)=q^{(k-1)(n-1)}(h_\lambda,e_{k-1}x)_{h_{n-1}}+q^{kn}(h_\lambda,e_kx)_{h_n}$ holds. 
\end{lem}
\begin{proof}
We utilize casework and the diagrammatic approach. There are two cases; either there exists a strand connecting $h_n$ and $e_k$, or there is not. 

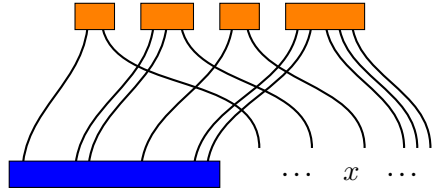
\begin{equation*}
\begin{tikzpicture}[scale=.7]
	\draw[thick] (1,0) [out=90, in=-90] to (2.25,3);
	\draw[thick] (2,0) [out=90, in=-90] to (3.5,3);
	\draw[thick] (2.25,0) [out=90, in=-90] to (3.75,3);
	\draw[thick] (3.25,0) [out=90, in=-90] to (5,3);
	\draw[thick] (4.25,0) [out=90, in=-90] to (6.25,3);
	\draw[thick] (4.5,0) [out=90, in=-90] to (6.5,3);
	\draw[thick] (5.5,.5) [out=90, in=-90] to (2.5,3);
	\draw[thick] (6.5,.5) [out=90, in=-90] to (4,3);
	\draw[thick] (7.5,.5) [out=90, in=-90] to (5.25,3);
	\draw[thick] (8.25,.5) [out=90, in=-90] to (6.75,3);
	\draw[thick] (8.5,.5) [out=90, in=-90] to (7,3);
	\draw[thick] (8.75,.5) [out=90, in=-90] to (7.25,3);
	\node at (6.25,0) {\ldots};
	\node at (7.25,0) {$x$};
	\node at (8.25,0) {\ldots};
	\filldraw[draw=black, fill=blue] (.75,-.25) rectangle (4.75,.25);
	\filldraw[draw=black, fill=orange] (2,2.75) rectangle (2.75,3.25);
	\filldraw[draw=black, fill=orange] (3.25,2.75) rectangle (4.25,3.25);
	\filldraw[draw=black, fill=orange] (4.75,2.75) rectangle (5.5,3.25);
	\filldraw[draw=black, fill=orange] (6,2.75) rectangle (7.5,3.25);
\end{tikzpicture}
\end{equation*}

If there exists a strand connecting $h_n$ to $e_k$, then summing across all possible diagrams, we obtain $(h_\lambda,e_{k-1}x)_{h_{n-1}}$. The insertion of $h_{n-1}$ is due to the fact that $n-1$ strands from $h_n$ intersect $x$, and must be accounted for when summing. However, each of the $n-1$ strands from the $h_n$ platform intersects each of the $k-1$ strands from $e_k$ to $h_\lambda$. This case contributes $q^{(k-1)(n-1)}(h_\lambda,e_{k-1}x)_{h_{n-1}}$.

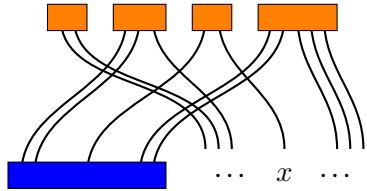
\begin{equation*}
\begin{tikzpicture}[scale=.7]
	\draw[thick] (1,0) [out=90, in=-90] to (3,3);
	\draw[thick] (1.25,0) [out=90, in=-90] to (3.25,3);
	\draw[thick] (2.25,0) [out=90, in=-90] to (4.5,3);
	\draw[thick] (3.25,0) [out=90, in=-90] to (5.75,3);
	\draw[thick] (3.5,0) [out=90, in=-90] to (6,3);
	\draw[thick] (4.5,.5) [out=90, in=-90] to (1.75,3);
	\draw[thick] (4.75,.5) [out=90, in=-90] to (2,3);
	\draw[thick] (5,.5) [out=90, in=-90] to (3.5,3);
	\draw[thick] (6,.5) [out=90, in=-90] to (4.75,3);
	\draw[thick] (7,.5) [out=90, in=-90] to (6.25,3);
	\draw[thick] (7.25,.5) [out=90, in=-90] to (6.5,3);
	\draw[thick] (7.5,.5) [out=90, in=-90] to (6.75,3);
	\node at (5,0) {\ldots};
	\node at (6,0) {$x$};
	\node at (7,0) {\ldots};
	\filldraw[draw=black, fill=blue] (.75,-.25) rectangle (3.75,.25);
	\filldraw[draw=black, fill=orange] (1.5,2.75) rectangle (2.25,3.25);
	\filldraw[draw=black, fill=orange] (2.75,2.75) rectangle (3.75,3.25);
	\filldraw[draw=black, fill=orange] (4.25,2.75) rectangle (5,3.25);
	\filldraw[draw=black, fill=orange] (5.5,2.75) rectangle (7,3.25);
\end{tikzpicture}
\end{equation*}

If no strand connects $h_n$ to $e_k$, then summing across all possible diagrams, we obtain $ (h_\lambda,e_{k}x)_{h_{n}}$. The insertion of $h_n$ is due to the fact that $n$ strands from $h_n$ intersect $x$, which must be accounted for in the summation. However, each of the $n$ strands from the $h_n$ platform intersects each of the $k$ strands from $e_k$ to $h_\lambda$, so this case contributes $q^{kn}(h_\lambda,e_kx)_{h_n}$. 

These are the only two possible cases and putting the two cases together yields the desired result. 
\end{proof}

Also note that 
\begin{displaymath}
    (h_n, e_k x)=
    \left\{
    \begin{array}{lr}
       (h_n,x) & \text{ if } k=0 \\
       (h_{n-1},x) & \text{ if } k=1 \\
       0 & \text{ if } k<0 \text{ or } k>1,
     \end{array}
     \right.
    \end{displaymath}
since at most one strand can connect $h_n$ and $e_k$. 
\end{proof}

We can now compute the general bilinear form $(h_{\lambda},e_k x)$, thereby facilitating the discovery of further relations between elementary symmetric functions. 

\begin{lem} We have that
\begin{equation*}
(h_\lambda, e_k x)=\sum\limits_{l=0}^{m}\sum\limits_{O_l^m}q^{|(\lambda_m^L-\sigma_l^m)(T_k^m-r(\sigma_l^m))|}(h_{\lambda_m^G},e_{k-l}x)_{h_{\lambda_m^L-\sigma_l^m}}.
\end{equation*}
\end{lem}
\begin{proof}
We induct on $m$. If $m=1$, then the proposition becomes: 
\begin{eqnarray}
(h_\lambda, e_k x)=\sum\limits_{l=0}^{1}\sum\limits_{O_l^1}q^{|(\lambda_1-\sigma_l^1)(k-r(\sigma_l^1))|}(h_{\lambda_1^G},e_{k-l}x)_{h_{\lambda_1^L-\sigma_l^1}},
\end{eqnarray}
which reduces to Proposition 1.2.

Now assume that the result holds for $m$. Then,
\begin{align*}
(h_\lambda, e_k x)&=\sum\limits_{l=0}^{m}\sum\limits_{O_l^m}q^{|(\lambda_m^L-\sigma_l^m)(T_k^m-r(\sigma_l^m))|}(h_{\lambda_m^G},e_{k-l}x)_{h_{\lambda_m^L-\sigma_l^m}}\\
&=\sum\limits_{l=0}^{m}\sum\limits_{O_l^m}q^{|(\lambda_m^L-\sigma_l^m)(T_k^m-r(\sigma_l^m))|}q^{(k-l-1)(\lambda_{m+1}-1)}(h_{\lambda_{m+1}^G},e_{k-l-1}x)_{h_{\lambda_m^L-\sigma_l^m}h_{\lambda_{m+1}-1}}\\
&+\sum\limits_{l=0}^{m}\sum\limits_{O_l^m}q^{|(\lambda_m^L-\sigma_l^m)(T_k^m-r(\sigma_l^m))|}q^{(k-l)(\lambda_{m+1})}(h_{\lambda_{m+1}^G},e_{k-l}x)_{h_{\lambda_m^L-\sigma_l^m}h_{\lambda_{m+1}}}.\\
\end{align*}

We therefore have
\begin{eqnarray}\label{Recursive Insertion}
&(h_\lambda, e_k x)=\sum\limits_{l=1}^{m+1}\sum\limits_{O_l^m}q^{|(\lambda_m^L-\sigma_{l-1}^m)(T_k^m-r(\sigma_{l-1}^m))|}q^{(k-l)(\lambda_{m+1}-1)}(h_{\lambda_{m+1}^G},e_{k-l}x)_{h_{\lambda_m^L-\sigma_{l-1}^m}h_{\lambda_{m+1}-1}}\\
&+\sum\limits_{l=0}^{m}\sum\limits_{O_l^m}q^{|(\lambda_m^L-\sigma_l^m)(T_k^m-r(\sigma_l^m))|}q^{(k-l)(\lambda_{m+1})}(h_{\lambda_{m+1}^G},e_{k-l}x)_{h_{\lambda_m^L-\sigma_l^m}h_{\lambda_{m+1}}}.
\end{eqnarray}

Let $a_l^m$ denote a composition in $O_l^m$ that ends in a $0$. Let $b_l^m$ denote a composition in $O_l^m$ that ends in a $1$. Let $A_l^m$ and $B_l^m$ be the set of all $a_l^m$ and $b_l^m$, respectively. Now, consider the terms indexed only by $1 \leq l \leq m$: 
\begin{align*}
&\sum\limits_{l=1}^{m}\sum\limits_{B_l^{m+1}}q^{|(\lambda_{m+1}^L-b_{l}^{m+1})(T_k^{m+1}-r(b_{l}^{m+1}))|}(h_{\lambda_{m+1}^G},e_{k-l}x)_{h_{\lambda_{m+1}^L-b_{l}^{m+1}}}\\
+&\sum\limits_{l=1}^{m}\sum\limits_{A_l^{m+1}}q^{|(\lambda_{m+1}^L-a_{l}^{m+1})(T_k^{m+1}-r(a_{l}^{m+1}))|}(h_{\lambda_{m+1}^G},e_{k-l}x)_{h_{\lambda_{m+1}^L-a_{l}^{m+1}}}\\
=&\sum\limits_{l=0}^{m+1}\sum\limits_{O_l^{m+1}}q^{|(\lambda_{m+1}^L-\sigma_l^{m+1})(T_k^{m+1}-r(\sigma_l^{m+1}))|}(h_{\lambda_{m+1}^G},e_{k-l}x)_{h_{\lambda_{m+1}^L-\sigma_l^{m+1}}}.
\end{align*}
The terms indexed by $1 \leq l \leq m$ match their corresponding terms in the Proposition. It remains to consider the cases $l=0$ and $l=m$. For $l=0$, note that there does not exist a $b_0^m$, and for $l=m+1$, note that there does not exist an $a_m^m$. From here, it is easy to see that these terms satisfy the proposition as well (the $l=0$ term can be found in the second sum of \eqref{Recursive Insertion} and the $l=m+1$ term can be found in the first term of \eqref{Recursive Insertion}). 
\end{proof}

Therefore, an explicit formula for the bilinear form can be given by: 

\begin{equation*}
    (h_\lambda,e_k x)=
        \left\{
    \begin{array}{lr}
       0 & \text{ if } k\geq z+1\\\\
       q^{|(\lambda_{k-1}^L-\sigma_{k-1}^{k-1})(T_k^{k-1}-r(\sigma_{k-1}^{k-1}))|}(h_{\lambda_{k-1}^L-\sigma_{k-1}^{k-1}}h_{\lambda_k-1}) & \text{ if } k=z\\\\
       \sum\limits_{O_{k-1}^{z-1}}q^{|(\lambda_{z-1}^L-\sigma_{k-1}^{z-1})(T_k^{z-1}-r(\sigma_{k-1}^{z-1}))|}(h_{\lambda_{z-1}^L-\sigma_{k-1}^{z-1}}h_{\lambda_z-1})\\
       +\sum\limits_{O_{k}^{z-1}}q^{|(\lambda_{z-1}^L-\sigma_{k}^{z-1})(T_k^{z-1}-r(\sigma_{k}^{z-1}))|}(h_{\lambda_{z-1}^L-\sigma_{k}^{z-1}}h_{\lambda_k}) & \text{ if } k<z.
       
     \end{array}
     \right.
\end{equation*}


\section{Conclusion and Further Research}
Through this work, we have contributed towards the program of oddification by studying properties of odd Dunkl operators in relation to diverse ideas in mathematics; namely, we connected odd Dunkl operators to odd divided difference operators, the classical Yang-Baxter equation, and the important Lie algebra $\sltwo$. We used inductive arguments and introduced refinements of the odd divided difference operators and the odd Dunkl operators in order to prove our main results. By discovering odd versions of the Dunkl Laplacian and $\sltwo$-triples, which play important roles in the representation theory of even symmetric polynomials, we have strengthened the odd theory and provided new areas of investigation for future researchers. 

In Section \ref{sl2}, we gave an action of $\sltwo$ on skew polynomials through a variant of the Khongsap-Wang odd Dunkl operator. In the future, we will try to describe the weight spaces and isotypic decomposition of this representation. We could also apply our results by studying higher degree differential operators in the odd case, since the representation theory of $\sltwo$ allows us to conveniently reduce degree to second order \cite{heckman}.

Ellis, one of the authors who introduced the odd nilHecke algebra, asked if there were odd analogs of other symmetric polynomials, such as Jack polynomials or Macdonald polynomials. Here, we outline a procedure for answering his question and making progress towards finding odd Jack polynomials. We first introduce the \emph{odd Cherednik operators} 
\begin{equation}
Y_i=-\alpha x_i\eta_i+\sum_{k<i} s_{i,k}-(n-1).
\end{equation}
Applying arguments similar to those used by Khongsap and Wang in \cite{kw}, we can find that 
\begin{center}
\begin{multicols}{2}
\begin{enumerate}
\item $Y_iY_j=Y_jY_i$
\item $s_iY_i=Y_{i+1}s_i-1$ 
\item $s_iY_{i+1}=Y_is_i+1$ 
\item $s_i Y_j=Y_js_i \text{ for } j\neq i,i+1$.
\end{enumerate}
\end{multicols}
\end{center}

The next step would be to find a scalar product for which the odd Cherednik operators are self-adjoint. One can then define the odd Jack polynomials as eigenfunctions of the odd Cherednik operators and study their properties as in \cite{knop}. Since the odd Cherednik operators are closely related to the odd Dunkl operators and the $r_{i,k}$ we introduced in Section \ref{odd dunkl}, the work in this paper would contribute significantly towards the study of odd Jack polynomials.

Factorization entails yet another problem of interest in the odd theory. For example, one can use the method of undetermined coefficients to show that, for odd $n$, 

\begin{equation*}
x_1^n-x_2^n=(x_1+ax_2)\sum_{k=0}^{n-1}v_ka^kx_1^{n-1-k}x_2^k,
\end{equation*}
where $v_{n-1}a^n=-1$, and $v_k$ is defined as follows: 
\begin{eqnarray*}
&v_k=\begin{cases}1&k\equiv 0,3 \text{ (mod }4) \\-1&k \equiv 1,2 \text{ (mod }4). \end{cases}
\end{eqnarray*}

Such identities arise in subtle ways in the action of operators on $P^-$ and the study of these kinds of noncommutative factorizations have separate combinatorial interests as well. 

In Section \ref{nh}, we introduced $q$-nilHecke algebras for all $q \neq 0,1,-1$. It would be interesting to study if the $q$-nilHecke algebras categorify an interesting Lie theoretic algebra, and whether they can be used to construct invariants of links or other geometric structures. One could also begin a diagrammatic study of the $q$-nilHecke algebras as in \cite{ekl}. 

In the same section, we defined elementary $q$-symmetric functions, which brings up the problem of finding relations between these generators and further studying the structure of $\sym^q_n$. We introduced a method for solving this problem using diagrams in Section \ref{diagrams}, and found some of these relations. However, the remaining relations between the $e_i$ are much more complex than their even or odd counterparts, and merit further study. When $q^3=1$, for example, the following degree $6$ relation holds:
\begin{align*}
&v_1 = e_{11211} + e_{12111} + e_{21111}\\
&v_2 = e_{1122} - 2e_{1221} + 3 e_{2112}+ e_{2211}\\
&v_3 = 2e_{1131}- 2e_{114}+ 2e_{1311}- 2e_{141} + 3 e_{222}+ 2e_{1113} - 2e_{411}\\
&v_1 + q^2v_2 + qv_3=0, 
\end{align*}
where $e_{\lambda}=e_{\lambda_1}\ldots e_{\lambda_k}$ for $\lambda=(\lambda_1, \lambda_2,\ldots,\lambda_k)$.

We also conjecture that our results, and especially the definition of odd Cherednik operators, have connections to recent geometric work of Braden, Licata, Proudfoot, and Webster, who have constructed category $\mathcal{O}$ for certain Cherednik algebras. As a result, we believe that the ideas in this paper will further develop the program of oddification and also create a more thorough understanding of higher representation theoretic structures.

\section{Acknowledgements}
I would like to thank my mentor, Alexander Ellis, for introducing me to his own research in representation theory and helping me obtain a more intuitive understanding of the odd construction. In addition, Professor Pavel Etingof, Professor Tanya Khovanova, and Dr. Ben Elias thoroughly edited this paper and provided encouragement. I also express gratitude to MIT PRIMES USA for giving me the opportunity to conduct this research.

\end{document}